\documentclass[hidelinks, 10pt]{amsart}
\usepackage{amsthm,amssymb,amscd,amsmath,amsfonts,amssymb,amscd,stmaryrd,hyperref,mathrsfs,yhmath,mathtools, amsrefs}
\usepackage[all]{xy}
\usepackage{hyperref}
\usepackage{enumerate}
\usepackage{physics}
\usepackage{ragged2e}
\usepackage{xfrac, faktor} 
\usepackage{lipsum}  
\usepackage{tikz-cd}
\usepackage{imakeidx}
\usepackage{subfiles}
\usepackage{bbm}
\usepackage{quiver}
\usepackage{mdframed}
\usepackage[margin = 1.2in]{geometry}
\usepackage{multicol}

\usepackage{amsmath,amsfonts,amscd,amssymb,verbatim,amsthm}

\newcommand{\HH}{\mathrm{H}}
\newcommand{\OO}{\mathcal{O}}

\DeclareMathOperator{\dRc}{dR,c}

\DeclareMathOperator{\JL}{JL}

\DeclareMathOperator{\fin}{fin}

\DeclareMathOperator{\Hom}{Hom}

\DeclareMathOperator{\LT}{LT}
\DeclareMathOperator{\Dr}{Dr}
\DeclareMathOperator{\Mod}{\mathbf{Mod}}
\DeclareMathOperator{\sm}{sm}
\DeclareMathOperator{\Gal}{Gal}
\DeclareMathOperator{\Nrd}{Nrd}

\DeclareMathOperator{\Char}{Char}

\DeclareMathOperator{\PicCon}{PicCon}
\DeclareMathOperator{\VectCon}{\mathbf{VectCon}}
\DeclareMathOperator{\Vect}{\mathbf{Vect}}
\DeclareMathOperator{\Coh}{\mathbf{Coh}}

\DeclareMathOperator{\Rep}{Rep}
\DeclareMathOperator{\End}{End}

\DeclareMathOperator{\reg}{reg}
\DeclareMathOperator{\ellip}{ell}

\DeclareMathOperator{\Irr}{Irr}
\DeclareMathOperator{\GL}{GL}
\DeclareMathOperator{\ab}{ab}
\DeclareMathOperator{\SL}{SL}
\DeclareMathOperator{\Aut}{Aut}
\DeclareMathOperator{\Stab}{Stab}

\DeclareMathOperator{\Spec}{Spec}

\DeclareMathOperator{\ESI}{ESI}

\DeclareMathOperator{\Der}{Der}

\makeatletter
\newcommand{\colim@}[2]{%
  \vtop{\m@th\ialign{##\cr
    \hfil$#1\operator@font colim$\hfil\cr
    \noalign{\nointerlineskip\kern1.5\ex@}#2\cr
    \noalign{\nointerlineskip\kern-\ex@}\cr}}%
}
\newcommand{\colim}{%
  \mathop{\mathpalette\colim@{\rightarrowfill@\textstyle}}\nmlimits@
}
\makeatother

\newcommand{\bA}{{\mathbb A}}

\newcommand{\bD}{{\mathbb D}}

\newcommand{\bP}{{\mathbb P}}
\newcommand{\bQ}{{\mathbb Q}}

\newcommand{\bZ}{{\mathbb Z}}

\newcommand{\sA}{{\mathcal A}}
\newcommand{\sB}{{\mathcal B}}
\newcommand{\sC}{{\mathcal C}}
\newcommand{\sD}{{\mathcal D}}
\newcommand{\sF}{{\mathcal F}}
\newcommand{\sG}{{\mathcal G}}

\newcommand{\sL}{{\mathcal L}}
\newcommand{\sM}{{\mathcal M}}
\newcommand{\sN}{{\mathcal N}}

\newcommand{\sR}{{\mathcal R}}
\newcommand{\sS}{{\mathcal S}}
\newcommand{\sT}{{\mathcal T}}

\newcommand{\sV}{{\mathcal V}}

\newcommand{\sY}{{\mathcal Y}}
\newcommand{\sW}{{\mathcal W}}

\newcommand{\fo}{{\mathfrak o}}
\newcommand{\fO}{{\mathfrak O}}

\newtheorem{thm}{Theorem}
\numberwithin{thm}{section}
\newtheorem{lemma}[thm]{Lemma} 
\newtheorem{prop}[thm]{Proposition} 
\newtheorem{cor}[thm]{Corollary} 
 
\newtheorem*{theoremA}{Theorem A}
\newtheorem*{theoremB}{Theorem B}
\newtheorem*{theoremC}{Theorem C}

\theoremstyle{definition}
\newtheorem{defn}[thm]{Definition}

\newtheorem{remark}[thm]{Remark}

\title{The Categories of Lubin-Tate and Drinfeld Bundles}
\date{\today}
\author{James Taylor}
\email{james.taylor@math.unipd.it}
\address{Dipartimento di Matematica, Universit\`{a} degli Studi di Padova, Via Trieste 63, 35131 Padova, Italy}

\begin{document}
\begin{abstract}
    For a finite extension $F$ of $\bQ_p$ and $n \geq 1$, we show that the category of Lubin-Tate bundles on the $(n-1)$-dimensional Drinfeld symmetric space is equivalent to the category of finite-dimensional smooth representations of the group of units of the division algebra of invariant $1/n$ over $F$.
\end{abstract}
\maketitle

\setcounter{tocdepth}{1}
\tableofcontents

\section{Introduction}

Let $p$ be a prime, let $F$ be a finite extension of $\bQ_p$, and let $n \geq 1$. The \emph{Drinfeld symmetric space} $\Omega$ of dimension $(n-1)$ is the rigid analytic space defined by removing from $\bP^{n-1}$ all $F$-rational hyperplanes, which carries a natural action of $G \coloneqq \GL_n(F)$ by restricting the natural action of $G$ on $\bP^{n-1}$. One reason to study $\Omega$ is that $\Omega$ is a natural source of interesting representations of $G$.

One such source of representations of $G$ is the category $\Vect^G(\Omega)$ of $G$-equivariant vector bundles on $\Omega$, and to this end different classes of equivariant vector bundles on $\Omega$ have been studied by many different authors \cite{ORL, LINDEN, JUNDR, KOH, AW, AW2, TAY2, TAY3}.

\subsection{Lubin-Tate Bundles}
Suppose from now on that $K$ is a complete extension of $\breve{F}$, the completion of the maximal unramified extension of $F$, and let us consider $\Omega$ as a rigid space over $K$. This paper concerns the full subcategory $\Vect^G_{\LT}(\Omega) \subset \Vect^G(\Omega)$ of \emph{Lubin-Tate bundles} introduced by Kohlhaase \cite{KOH}. Concretely, an object $\sV \in \Vect^G(\Omega)$ is Lubin-Tate if for some $m \geq 1$ the $\OO(\sN_m)$-module
\[
\OO(\sN_m) \otimes_{\OO(\Omega)} \sV(\Omega)
\]
is generated by $G_m$-invariant elements, where $\sN_m$ is the level $m$ Drinfeld covering of $\Omega$ of finite degree and $G_m \coloneqq 1 + \pi^m M_n(\OO_F)$ is the $m$th congruence subgroup of $G$.

From the construction of the category $\Vect^G_{\LT}(\Omega)$, there is an equivalence
\[
\bD_{\Dr} \colon \Vect^H_{\Dr}(\bP^{n-1}) \xrightarrow{\sim} \Vect^G_{\LT}(\Omega)
\] 
where $\Vect^H_{\Dr}(\bP^{n-1}) \subset \Vect^H(\bP^{n-1})$ is the full subcategory of \emph{Drinfeld bundles} on $\bP^{n-1}$ and $H$ is the group of units of the division algebra over $F$ of invariant $1/n$. This equivalence allows one to translate problems regarding $G$-equivariant vector bundles on $\Omega$ to problems regarding $H$-equivariant vector bundles on $\bP^{n-1}$. For example, Kohlhaase defined a fully faithful inclusion functor
\[
\Psi_H \coloneqq \bD_{\Dr}(\OO_{\bP^{n-1}} \otimes_K -) \colon \Rep_{K}^{\sm}(H) \rightarrow \Vect^G_{\LT}(\Omega),
\]
from the category of finite-dimensional smooth representations of $H$ over $K$. This was used by Dospinescu and Le Bras to translate problems about $\Psi_H(V)$ into problems about $\OO_{\bP^{n-1}} \otimes_K V$ in their proof \cite[\S 10]{DLB} that the locally analytic $\GL_2(\bQ_p)$-representation $\Gamma(\Omega, \Psi_H(V))^*$ is coadmissible whenever $V$ is irreducible and $\dim(V) > 1$.

The functor $\Psi_H$ provides examples of objects of $\Vect^G_{\LT}(\Omega)$, but a priori $\Vect^G_{\LT}(\Omega)$ is much larger. Our main result is that, somewhat surprisingly, $\Psi_H$ is actually an equivalence.

\begin{theoremA}
    The functor
    \[
    \Psi_H \colon \Rep_{K}^{\sm}(H) \rightarrow \Vect^G_{\LT}(\Omega)
    \]
    is an equivalence of categories.
\end{theoremA}

In particular, any Drinfeld bundle on $\bP^{n-1}$ is of the form $\OO_{\bP^{n-1}} \otimes_K V$. We also note that there is an asymmetry here: the analogously defined functor $\Psi_G$ is not an equivalence when $n \geq 2$ (Remark \ref{rem:PsiGnotequivingeneral}).

The main idea behind the proof of Theorem A is an equivalent characterisation of the notation of a Lubin-Tate bundle (Lemma \ref{lem:equivdefs}), from which the fact that $\Psi_H$ is an equivalence becomes relatively formal (cf.\ Corollary \ref{cor:maincorLT} and Theorem \ref{thm:finalmainthmLT}). The main observation is that for any Lubin-Tate bundle $\sV$, the $G_m$-invariants and $G^m$-invariants of the $\OO(\sN_m)$-module $\OO(\sN_m) \otimes_{\OO(\Omega)} \sV(\Omega)$ considered above actually coincide, where $G^m = \det^{-1}(1 + \pi^m \OO_F)$, the normal closure of $G_m$ in $G$.

\subsection{$G^0$-Finite Bundles}

The functor $\Psi_H$ also lands in the full subcategory 
\[
    \Vect^G(\Omega)_{G^0\text{-}\fin} \subset \Vect^G(\Omega),
\]
of $G^0$-finite bundles, where $G^0 = \ker(\nu \circ \det \colon G \rightarrow \bZ)$ (see Definition \ref{defn:finiteobj}). In our previous work \cite{TAY3} we showed that the natural factorisation of $\Psi_H$ through the category $\VectCon^G(\Omega)$ of $G$-equivariant vector bundles with connection on $\Omega$ induces an equivalence
\[
\Psi_H^{\nabla} \colon \Rep_{K}^{\sm}(H) \xrightarrow{\sim} \VectCon^G(\Omega)_{G^0\text{-}\fin}
\]
to the analogously defined full subcategory $\VectCon^G(\Omega)_{G^0\text{-}\fin}$ of $\VectCon^G(\Omega)$.

Using some results we establish relating the infinitesimal action of $G$ and the action of the tangent sheaf on the covering spaces $\sN_m$ (Lemma \ref{lem:groupgeneratestangentsheaf}), we adapt the techniques from \cite{TAY3} to also show that $\Psi_H$ is also an equivalence onto $\Vect^G(\Omega)_{G^0\text{-}\fin}$.
\begin{theoremB}
    The functor $\Psi_H$ induces an equivalence of categories
    \[
        \Psi_H \colon \Rep_{K}^{\sm}(H) \xrightarrow{\sim} \Vect^G(\Omega)_{G^0\text{-}\fin}.
    \]
\end{theoremB}
    \noindent In particular, this shows that there is an equality of full subcategories
    \[
        \Vect^G_{\LT}(\Omega) = \Vect^G(\Omega)_{G^0\text{-}\fin}
    \]
    inside $\Vect^G(\Omega)$, and that the restriction functor
    \[
        \VectCon^G(\Omega)_{G^0\text{-}\fin} \rightarrow \Vect^G(\Omega)_{G^0\text{-}\fin}
    \]
    is an equivalence of categories.

\begin{remark}
	We further show that $\Vect^G(\Omega)_{G^0\text{-}\fin}$ is closed under sub-objects (Theorem \ref{thm:finiteVB1}), and so in particular that $\Psi_H$ preserves irreducibility. This generalises the results of Dospinescu and Le-Bras \cite[\S 10.1]{DLB} established during their proof of the coadmissibility of $\Gamma(\Omega, \Psi_H(V))^*$ described above, from $\GL_2(F)$ to $\GL_n(F)$.
\end{remark}

\begin{remark}
    Theorem A and Theorem B show that objects of $\Vect^G_{\LT}(\Omega)$ are functorially equipped with a connection, which for $W \in \Rep_K^{\sm}(H)$ agrees with the connection on $\Psi_H(W)$ considered in \cite[Rem.\ 4.8]{KOH} and the connection on $\Psi_H(W)$ which we consider in Section \ref{sect:DRside}.
\end{remark}

\subsection{Other Results}

We also give a complete description of the category of $G_0$-equivariant Lubin-Tate bundles $\Vect^{G_0}_{\LT}(\Omega)$ considered by Kohlhaase (Theorem \ref{thm:mainthmLT0}). Using this, we give a precise description of which of these Lubin-Tate bundles arise from smooth representations of $G_0 = \GL_n(\OO_F)$ and $H_0 = \OO_D^{\times}$, and how these subcategories intersect (Theorem \ref{thm:intersectEI}).

The proof of Theorem B used $\sD$-modules in an essential way, even though the statement itself does not involve them. As another application of the use of $\sD$-modules to the study of $\Psi_H$, we prove:

\begin{theoremC}
    For any smooth character $\chi \colon F^\times \rightarrow K^\times$ and $V \in \Rep_{K}^{\sm}(H)$,
    \[
    \Psi_H(V \otimes \chi_H) = \Psi_H(V) \otimes \chi_G.
    \]
\end{theoremC}

Here $\chi \mapsto \chi_G = \chi \circ \det$ and $\chi \mapsto \chi_H = \chi \circ \Nrd$ are the canonical identifications of the smooth character groups of $G, H$ and $F^\times$. Our proof of Theorem C uses in an essential way the enrichment of $\Psi_H$ to $\Psi_H^{\nabla}$, and makes use of a result of Kohlhaase on the trace of the action of elliptic elements of $G$ on certain stalks of $\Psi_H(V)$ \cite[Thm.\ 4.7]{KOH}.

\subsection*{Outline of the Paper}
In Section \ref{sect:smoothSLreps} and Section \ref{sect:generalities} we prove some general technical results regarding equivariant vector bundles arising from certain semilinear representations, which we use in many places throughout the rest of the paper. We advise the reader to skip these sections and refer back to them when necessary. In Section \ref{sect:LTDRbundles} we relate the infinitesimal action of $G$ on Drinfeld spaces with the action of the tangent sheaf. In Section \ref{sect:DRside0} we describe the category of $G_0$-equivariant Lubin-Tate bundles (Theorem \ref{thm:mainthmLT0}), and use this description to establish properties of $\Vect^{G_0}_{\LT}(\Omega)$. In Section \ref{sect:DRside} we prove Theorem A that $\Psi_H$ is an equivalence onto the category of $G$-equivariant Lubin-Tate bundles (Theorem \ref{thm:finalmainthmLT}). In Section \ref{sect:G0finite} we prove Theorem B and show that the essential image of $\Psi_H$ is also equal to the category of $G^0$-finite bundles (Theorem \ref{thm:finiteVB2}), before establishing some properties of $\Psi_H$ in Section \ref{sect:properties} including Theorem C (Theorem \ref{thm:twists}). 

\subsection*{Conventions} All group representations considered in this paper are assumed to be finite-dimensional.

\subsection*{Acknowledgements}

I would like to thank Jan Kohlhaase, Nicolas Dupr\'{e}, Tom Adams, Konstantin Ardakov and Alex Horawa for their comments / interest in this work. In particular I am very grateful to Jan Kohlhaase for the invitation to Essen to discuss this work, and his suggestions for how one might prove Theorem A avoiding the use of $\sD$-modules, which greatly simplified the proof of the main result. I would also like to thank the two anonymous referees for their comments which all improved the paper. This research was supported by an LMS Early Career Research Fellowship at the University of Cambridge and the departmental grant \textit{Progetto Sviluppo Dipartimentale} - UNIPD PSDIP23O88 at the University of Padova.

\section{Smooth Semilinear Representations}\label{sect:smoothSLreps}

We are interested in various categories of semilinear representations, which we now introduce. Later, we will show how these categories are related to the category of Lubin-Tate bundles.

\begin{defn}
For a group $G$ which acts on a ring $R$, and a subgroup $H$ of $G$, we write 
\[
\Rep^H_R(G)
\]
for the full subcategory of modules $V$ over the skew group ring $R \rtimes G$ such that
\begin{enumerate}
    \item $V$ is free of finite rank over $R$,
    \item $R \cdot V^H = V$.
\end{enumerate}
\end{defn}

When $H$ is trivial, we write this as $\Rep_R(G)$. Note that when the action of $H$ on $R$ is trivial, condition (2) is simply that $H$ acts trivially on $V$. We also have the following topological version.

\begin{defn}\label{def:smoothrep}
If $G$ is a topological group which acts on a ring $R$, then we write 
\[
\Rep^{\sm}_R(G)
\]
for the full subcategory of modules $V$ over the skew group ring $R \rtimes G$ such that
\begin{enumerate}
    \item $V$ is free of finite rank over $R$,
    \item there exists an open subgroup $H$ of $G$ such that $R \cdot V^H = V$.
\end{enumerate} 
\end{defn}
When the action of $G$ on $R$ is smooth (meaning $R = \cup_{H \leq_{o} G} R^H$, as $H$ varies over the open subgroups of $G$), then for $V$ free of finite rank over $R$, $V \in \Rep^{\sm}_R(G)$ if and only if $V = \cup_{H \leq_{o} G} V^H$. In particular, when the action of $G$ on $R$ is trivial, $\Rep^{\sm}_R(G)$ is simply the category of smooth representations of $G$ on free $R$-modules of finite rank.

\subsection{Some Useful Lemmas}

We are principally interested in the case when $R$ is a product of fields, in which case we shall make use of the following two lemmas.

\begin{lemma}\label{lem:autpresprincipalidem}
    Suppose that $\phi \colon L \rightarrow L$ is a ring automorphism of $L$, where $L = \prod_i L_i$ is a product of fields. Then $\phi$ preserves the set $\{e_i\}_i$ of principal idempotents.
\end{lemma}

\begin{proof}
    Idempotents are preserved by automorphisms, and principal idempotents of $L$ can be characterised as those $e \in \text{Idem}(L)$ for which $ef = e$ or $ef = 0$ for any $f \in \text{Idem}(L)$.
\end{proof}

In particular, any group action on $L = \prod_i L_i$ by ring automorphisms will preserve $\{e_i\}_i$.

\begin{lemma}\label{lem:modoverskewisfree}
    Suppose that $L = \prod_i L_i$ is a product of fields with an action of a group $G$ by ring automorphisms, which acts on the set $\{e_i\}_i$ of principal idempotents of $L$ transitively. Suppose that $M$ is a $L \rtimes G$-module for which the natural map
    \[
    M \rightarrow \prod_i e_i \cdot M, \qquad m \mapsto (e_i \cdot m)_i
    \]
    is an isomorphism and $\dim_{L_i} e_i \cdot M < \infty$ for some $i$. Then $M$ is free of finite rank over $L$.
\end{lemma}

\begin{proof}
    Fix some $e_0 \in \{e_i\}_i$ with $\dim_{L_0} e_0 \cdot M < \infty$, and for each $i$ fix some $g_i \in G$ with $g_i(e_0) = e_i$. This allows us to define field isomorphisms
    \[
    \gamma_i \coloneqq \pi_i \circ g_i \circ \iota_0 \colon L_0 \xrightarrow{\sim} L_i,
    \]
    where $\pi_i \colon L \rightarrow L_i$ is the projection and $\iota_0 \colon L_0 \rightarrow L$ is defined by setting $\iota_0(\lambda)$ to be the unique element of $L$ with $\pi_0(\iota_0(\lambda)) = \lambda$ and $\pi_j (\iota_0(\lambda)) = 0$ when $j \neq 0$.
    
    The map $g_i \colon M \rightarrow M$ restricts to a bijection $e_0 \cdot M \xrightarrow{\sim} e_i \cdot M$, which is $L_0$-linear with respect to the natural action of $L_0$ on $e_0 \cdot M$ and the action of $L_0$ on $e_i \cdot M$ through $\gamma_i \colon L_0 \xrightarrow{\sim} L_i$. Then if $\{f_m^0\}_m$ is a basis for $e_0 \cdot M$ over $L_0$, the set $\{f_m\}_m$ defined by $f_m = (g_i(f_m^0))_i$ is a basis of $\prod_i M_i$ over $L$. Indeed, because each $\{g_i(f_m^0)\}_m$ is a basis of $e_i \cdot M$ over $L_i$, the set $\{f_m\}_m$ is linearly independent and (using crucially that $\{f_m\}_m$ is finite) the set $\{f_m\}_m$ also spans $M$.
\end{proof}

\begin{remark}
    If the assumption that some $\dim_{L_i} e_i \cdot M < \infty$ is removed, then the same proof shows that $M$ is free (but potentially of infinite rank), provided that $|I| < \infty$.
\end{remark}

\section{Semilinear Representations and Equivariant Vector Bundles}\label{sect:generalities}

Suppose that $X$ is a scheme or rigid space over a field $K$ with an action of a group $G$, and that $H$ is a subgroup of $G$. In this context we can consider the category $\Vect^G(X)$ of $G$-equivariant vector bundles on $X$, and when $X$ is smooth the category $\VectCon^G(X)$ of $G$-equivariant vector bundles with connection on $X$ (see, for example, \cite[\S 2.8]{TAY3}).

We consider the quadruple $(\sA, \sB, \sC, L)$ as in one of the following two cases:
\begin{enumerate}
	\item[\textbf{(A)}]
    $\sA = \VectCon^G(X)$, $\sB = \VectCon^H(X)$, $\sC = \VectCon(X)$, $L = c(X)^H$,
	\item[\textbf{(B)}] 
	$\sA = \Vect^G(X)$, $\sB = \Vect^H(X)$, $\sC = \Vect(X)$, $L = \OO(X)^H$,
\end{enumerate}
where we additionally assume in case (A) that $K$ has characteristic $0$ and $X$ is smooth over $K$, in order for the quantities $(\sA, \sB, \sC, L)$ to be well-defined. Here $c(X)$ is the $K$-algebra of global sections of the sheaf of constant functions as considered in \cite[\S 3]{TAY3}. In either case we further assume:
\begin{enumerate}
    \item $X$ is the categorical disjoint union of its connected components,
    \item The action of $G$ on $\OO(X)$ preserves $L$,
    \item $G$ acts transitively on $\pi_0(X)$.
\end{enumerate}
\begin{remark}
We note that the requirement that $X$ is the categorical disjoint union of its connected components is automatic if $X$ is a rigid space \cite[\S 2.1]{CON}, or if $X$ is in case (A) as smooth schemes are locally noetherian and thus locally connected \cite[Lem.\ 3.4(1)]{TAY3}. Furthermore, if $H$ is normal in $G$ or if $L = c(X)$ (in case (A)) then the requirement that $G$ preserves $L$ is automatic. Additionally, the results of this section still have consequences when the assumption that $G$ acts transitively on $\pi_0(X)$ is dropped, as they can be applied to each orbit of the action of $G$ on $X$.
\end{remark}
Write $\{X_i\}_{i \in \pi_0(X)}$ for the set of connected components of $X$, and let $\mathfrak{O}$ be the set of $H$ orbits on $\pi_0(X)$. For each $\fo \in \fO$, set
\[
X_{\fo} = \bigsqcup_{i \in \fo} X_i.
\]
For $\fo \in \fO$ and $i \in \pi_0(X)$ we write
\begin{enumerate}
	\item[\textbf{(A)}]
	$\sB_{\fo} = \VectCon^H(X_{\fo})$, $\sB_{i} = \VectCon^{H_i}(X_{i})$,
	\item[\textbf{(B)}] 
	$\sB_{\fo} = \Coh^H(X_{\fo})$, $\sB_{i} = \Coh^{H_i}(X_{i})$,
\end{enumerate}
where $H_i = \Stab_H(X_i)$ and $\Coh^H(X_{\fo})$ is the category of $H$-$\OO_{X_{\fo}}$-modules for which the underlying $\OO_{X_{\fo}}$-module is coherent (and similarly for $\Coh^{H_i}(X_i)$). We note that these definitions are uniform, as in case (A) a $\sD$-module on $X_{\fo}$ which is coherent as an $\OO_{X_{\fo}}$-module is automatically locally free \cite[Lem.\ 2.38]{TAY3}. With this notation, we also make the additional assumption that:
\begin{enumerate}
    \item[(4)] Each $\OO_{X_{\fo}}$ is irreducible as an object of $\sB_{\fo}$.
\end{enumerate}
We note that by \cite[Ex.\ 2.33]{TAY3}, this is equivalent to the assumption that
\begin{enumerate}
    \item[(4')] Each $\OO_{X_{i}}$ is irreducible as an object of $\sB_{i}$.
\end{enumerate}
As $X$ is the disjoint union of the $X_{\fo}$, 
\[
L = \prod_{\fo \in \fO} L_{\fo}
\]
where for $\fo \in \fO$ and $i \in \pi_0(X)$ we write
\begin{enumerate}
	\item[\textbf{(A)}]
	$L_{\fo} = c(X_{\fo})^H$, $L_{i} = c(X_{i})^{H_i}$
	\item[\textbf{(B)}] 
	$L_{\fo} = \OO(X_{\fo})^H$, $L_{i} = \OO(X_{i})^{H_i}$.
\end{enumerate}
\begin{lemma}\label{lem:Lisfield}
For $i \in \fo$ and $\fo \in \fO$, $L_{\fo}$ and $L_{i}$ are fields and the projection map
\[
\pi \colon L_{\fo} \xrightarrow{\sim} L_i.
\]
is an isomorphism.
\end{lemma}

\begin{proof}
The projection $\pi$ forms part of the commutative diagram
\[\begin{tikzcd}
	{L_{\fo}} & {L_i} \\
	{\End_{\sB_{\fo}}(\OO_{X_{\fo}})} & {\End_{\sB_{i}}(\OO_{X_{i}})}
	\arrow["\pi", from=1-1, to=1-2]
	\arrow["\sim"', from=1-1, to=2-1]
	\arrow["\sim"', from=1-2, to=2-2]
	\arrow[from=2-1, to=2-2]
\end{tikzcd}\]
with vertical isomorphisms given by taking invariants of the isomorphism of \cite[Lem.\ 3.3]{TAY3}. Each endomorphism ring is a division ring by assumption (4) and Schur's Lemma, hence both $L_{\fo}$ and $L_i$ are fields. Furthermore, $\pi$ is an isomorphism as the lower map is an isomorphism by \cite[Ex.\ 2.33]{TAY3}.
\end{proof}

\subsection*{The Functor $\OO_X \otimes_{L} -$}
We consider the functor
\[
\OO_X \otimes_{L} - \colon \Rep^H_L(G) \rightarrow \sA
\]
defined by considering $\OO_X \otimes_L V$ as a $G$-equivariant sheaf via the diagonal action of $G$. In case (A), for $V \in \Rep^H_L(G)$, we additionally view $\OO_X \otimes_L V$ as $\sD_X$-module via the action of $\sD_X$ on the left factor, which induces a well-defined action on the tensor product as the tangent sheaf $\sT_X$ acts trivially on $L \subset c(X)$, and further gives $\OO_X \otimes_{L} V$ the structure of a $G$-$\sD_X$-module.

\subsection*{The Functor $\Hom_{\sB}(\OO_X, -)$}
In the other direction we have the functor
\[
\Hom_{\sB}(\OO_X, -) \colon \sA \rightarrow \Mod_L.
\]
Here, for $\sV \in \sA$, $\Hom_{\sB}(\OO_X, \sV)$ is naturally a module over $L$ as $L$ is fixed by $H$. In order to see that this is free of finite rank over $L$, we use the following lemma.

\begin{lemma}\label{lem:injective}
Suppose that $\sV \in \sA$. Then the natural morphism
    \begin{equation}\label{eqn:injectivemap}
        \OO_X \otimes_L \Hom_{\sB}(\OO_X, \sV) \rightarrow \sV
    \end{equation}
in $\sB$ is injective, and $\Hom_{\sB}(\OO_X, \sV)$ is free of finite rank over $L$.
\end{lemma}

\begin{proof}
For injectivity, as $X$ is a $H$-stable disjoint union of $\{X_{\fo}\}_{\fo}$, the morphism (\ref{eqn:injectivemap}) factors as the product of morphisms
\[
\OO_{X_{\fo}} \otimes_{L_{\fo}} \Hom_{\sB_{\fo}}(\OO_{X_{\fo}}, \sV_{\fo}) \rightarrow \sV_{\fo},
\]
where $\sV_{\fo} \coloneqq \sV|_{X_{\fo}}$, and we are reduced to showing that each is injective, for which we modify the proof of \cite[Lem.\ 4.3]{TAY3}. Recall that $L_{\fo}$ is a field by Lemma \ref{lem:Lisfield}, and suppose that $f_1, ... ,f_k \in \Hom_{\sB_{\fo}}(\OO_{X_{\fo}}, \sV_{\fo})$ are $L_{\fo}$-linearly independent. Define $e_1, ... , e_k \in \sV_{\fo}(X_{\fo})^H$ by $e_i \coloneqq f_i(1_{X_{\fo}})$ which are also $L_{\fo}$-linearly independent because the natural map
\[
\Hom_{\sB_{\fo}}(\OO_{X_{\fo}}, \sV_{\fo}) \rightarrow \sV(X_{\fo})^H, \qquad f \mapsto f(1_{X_{\fo}})
\]
is $L_{\fo}$-linear and injective. Because $\OO_{X_{\fo}}$ is irreducible in $\sB_{\fo}$ by assumption (4), it is sufficient for us to show that the sum
\[
		\sum_{i = 1}^k \OO_{X_{\fo}} \cdot e_i \hookrightarrow \sV_{\fo}
\]
is direct. We prove this by induction on $k \geq 1$. When $k = 1$ this is trivially true, so suppose that the statement is true for some fixed $k \geq 1$, and consider
\[
		\sum_{i = 1}^{k+1} \OO_{X_{\fo}} \cdot e_i.
\]
After rearranging the factors if necessary, it is sufficient to show that
\[
	\left( \bigoplus_{i = 1}^{k} \OO_{X_{\fo}} \cdot e_i \right) \bigcap \OO_{X_{\fo}} \cdot e_{k+1} = 0.
\]
If this intersection were non-zero, then 
\[
	\left( \bigoplus_{i = 1}^{k} \OO_{X_{\fo}} \cdot e_i \right) \bigcap  \OO_{X_{\fo}} \cdot e_{k+1} = \OO_{X_{\fo}} \cdot e_{k+1},
\]
	by the irreducibility of $\OO_{X_{\fo}} \cdot e_{k+1}$. We can therefore write
 \[
 e_{k+1} = \lambda_1 e_1 + \cdots + \lambda_k e_k
 \]
for unique $\lambda_j \in \OO(X_{\fo})$. For any $h \in H$,
\[
0 = h(e_{k+1}) - e_{k+1} = (h(\lambda_1) - \lambda_1) e_1 + \cdots + (h(\lambda_k) - \lambda_k) e_k,
\]
and therefore each $\lambda_j \in \OO(X_{\fo})^H$ by induction. Furthermore, in case (A), for any affinoid (resp. affine) open subset $U$ and $\partial \in \sT(U)$,
	\begin{align*}
		0 &= \partial(e_{k+1}) = \partial(\lambda_1 e_1 + \cdots + \lambda_k e_k), \\
		&= \partial(\lambda_1)e_1 + \cdots + \partial(\lambda_k) e_k,
	\end{align*}
	in $\sV(U)$, and therefore by induction $\partial(\lambda_i) = 0$ for all $i = 1, ... ,n$. As this holds for each admissible open subset $U \subset X_{\fo}$, each $\lambda_i \in c(X_{\fo})$ by \cite[Lem.\ 3.2]{TAY3}. Therefore, in either case (A) or (B), $\lambda_i \in L_{\fo}$, which is a contradiction as we know that $e_1, ... ,e_k, e_{k+1}$ are linearly independent over $L_{\fo}$.

Finally, the fact that $\Hom_{\sB}(\OO_X, \sV)$ is free of finite rank over $L$ follows directly from Lemma \ref{lem:modoverskewisfree}, which uses Lemma \ref{lem:Lisfield}, that
\[
\Hom_{\sB}(\OO_{X}, \sV) \xrightarrow{\sim} \prod_{\fo} \Hom_{\sB_{\fo}}(\OO_{X_{\fo}}, \sV_{\fo}),
\]
that each $\dim_{L_{\fo}} \Hom_{\sB_{\fo}}(\OO_{X_{\fo}}, \sV_{\fo}) \leq \rank_{X_{\fo}}(\sV_{\fo}) < \infty$, and that $G$ acts transitively on the orbits $\fo$.
\end{proof}

When $H$ is normal in $G$, the natural action of $G$ on $\Hom_{\sC}(\OO_X, \sV)$ \cite[Rem.\ 2.32]{TAY3} preserves $\Hom_{\sB}(\OO_X, \sV)$ because $H$ is normal in $G$. This gives $\Hom_{\sB}(\OO_X, \sV)$ the structure of a $L \rtimes G$-module upon which $H$ acts trivially, hence the functor $\Hom_{\sB}(\OO_X, -)$ factors as
\[
 \Hom_{\sB}(\OO_X, -) \colon \sA \rightarrow \Rep_L^H(G),
\]
and (\ref{eqn:injectivemap}) is a morphism in $\sA$.

We can now state the main result of this section.

\begin{prop}\label{prop:mainprop}
Suppose that $X$, $G$ and $H$ are as described at the start of Section \ref{sect:generalities} and satisfy the assumptions (1), (2), (3) and (4). Then:
\begin{enumerate}
    \item
    The functor
    \[
    \OO_X \otimes_{L} - \colon \Rep^H_L(G) \rightarrow \sA
    \]
    is exact, monoidal, and fully faithful.
    \item
    Suppose that $H$ is normal in $G$. Then the essential image is the full subcategory of objects $\sV \in \sA$ for which the injective map \emph{(\ref{eqn:injectivemap})} is an isomorphism, and on this full subcategory
    \[
    \Hom_{\sB}(\OO_X, -) \colon \sA \rightarrow \Rep^H_L(G)
    \]
    is a quasi-inverse to $\OO_X \otimes_L -$.
    \item 
    If each $X_i$ is quasi-Stein then the essential image of $\OO_X \otimes_{L} -$ is closed under sub-quotients. 
    \end{enumerate}
\end{prop}

For simplicity, here and in the proof of Proposition \ref{prop:mainprop} we use the term \emph{quasi-Stein} to mean \emph{quasi-Stein} of finite dimension ($\sup_{x \in X} \dim(\OO_{X,x}) < \infty$) when $X$ is a rigid space, and to mean \emph{affine} when $X$ is a scheme. 

\begin{remark}
    In case (A), the essential image description (3) is equivalently those $\sV \in \sA$ for which $\rank_{L}(\Hom_{\sB}(\OO_X, \sV)) = \rank(\sV)$, which can be seen using the same argument as \cite[Thm.\ 4.4(3)]{TAY3}.
\end{remark}

\begin{remark}
 Lemma \ref{lem:injective} and Proposition \ref{prop:mainprop} generalise Lemma 4.3 and Theorem 4.4 of \cite{TAY3}, which can be seen by taking the triple $(X, G, H)$ to be $(X, G \times H, G)$ (in the notation of loc.\ cit.).
\end{remark}

\begin{proof}
    For point (1), first note that $\OO_X \otimes_L -$ is clearly exact, monoidal and faithful. To show that it is full, suppose that $V,W \in \Rep_L^H(G)$, and
    \[
    f \colon \OO_X \otimes_L V \rightarrow \OO_X \otimes_L W
    \]
    is a morphism in $\sA$. For any $v \in V$, because $H$ acts trivially on $V$ and $W$ and $f$ is $H$-equivariant,
    \[
    f_X(1_X \otimes v) \in \OO(X)^H \otimes_L W.
    \]
    Furthermore, in case (A), for any admissible open subset $U \subset X$ and $\partial \in \sT(U)$
    \[
    \partial (f_U(1_U \otimes v)) = f_U(\partial(1_U) \otimes v) = 0,
    \]
    and therefore
    \[
    f_X(1_X \otimes v) \in c(X) \otimes_L W
    \]
    by \cite[Lem.\ 3.2]{TAY3}. In either case (A) or (B), 
    \[
    f_X(1_X \otimes v) \in L \otimes_L W \equiv W,
    \]
    and this allows us to define a morphism $\lambda \colon V \rightarrow W$ with $f_X(1_X \otimes v) = 1 \otimes \lambda(v)$. It is direct to see, as $f$ is a morphism in $\sA$, that $f = 1 \otimes \lambda$ and hence $\OO_X \otimes_L -$ is full.

    For point (2), first note that because $H$ is normal in $G$ we may view $\Hom_{\sB}(\OO_X, -)$ as a functor from $\sA$ to $\Rep_L^H(G)$, not just $\Mod_L$, as described above. If $\sV \in \sA$ and the injective map (\ref{eqn:injectivemap}) is an isomorphism then $\sV$ is in the essential image of the functor $\OO_X \otimes_L -$, and on the full subcategory of such $\sV$ the functor $\Hom_{\sB}(\OO_X, -)$ provides a right quasi-inverse to $\OO_X \otimes_L -$. Conversely, if $\sV = \OO_X \otimes_L V$ is the in the essential image, then 
    \[
    \Hom_{\sB}(\OO_X, \OO_X \otimes_L V) = \Hom_{\sB}(\OO_X \otimes_L L, \OO_X \otimes_L V) \xleftarrow{\sim} \Hom_{L[H]}(L, V) \equiv V,
    \]
    by the fully faithfulness of point (1) in the case when we take $G = H$. Therefore, 
    \[
    \OO_X \otimes_L \Hom_{\sB}(\OO_X, \sV) \rightarrow \sV
    \]
    is an isomorphism, being identified with the identity map
    \[
    \OO_X \otimes_L V \rightarrow \sV
    \]
    and we see that $\Hom_{\sB}(\OO_X, -)$ provides a left quasi-inverse to $\OO_X \otimes_L -$.
    
    For point (3), in case (A) one can use the same argument as \cite[Thm.\ 4.4(4)]{TAY3}, which uses point (2) and the fact that $\VectCon^G(X)$ is closed under quotients in $\Mod(G\text{-}\sD_X)$. However it isn't true in general that $\Vect^G(X)$ is closed under quotients in $\Mod(G \text{-}\OO_X)$, and so we give an alternative proof. First note that because the functor is exact and full, it is sufficient to show that the essential image is closed under sub-objects. Given $V \in \Rep^H_L(G)$, suppose that $\sF \subset \OO_X \otimes_L V$ is a sub-object in $\sA$, and define
    \[
		W \coloneqq \{v \in V \mid s \otimes v \in \sF(U) \text{ for any open } U \subset X, s \in \OO_X(U)\}.
	\]
    This is an $L$-submodule of $V$ as the tensor product is over $L$, and further a $L \rtimes G$ 
    -submodule of $V$, as given $g \in G$ and $v \in V$, then for any open subset $U \subset X$ and $s \in \OO_X(U)$,
    \[
    s \otimes g(v) = g^{\OO_X}_{g^{-1}(U)}((g^{-1})^{\OO_X}_U(s)) \otimes g(v) = g^{\OO_X \otimes V}((g^{-1})^{\OO_X}_U(s) \otimes v) \in \sF(U).
    \]
    We now claim that the injection $\OO_X \otimes_L W \hookrightarrow \sF$ is an isomorphism. It is sufficient to show that for all $i \in \pi_0(X)$, the restriction to $X_i$ is surjective, or equivalently that the morphism on sections
    \[
    \OO(X_i) \otimes_{L} W \hookrightarrow \sF(X_i)
    \]
    is surjective, as each $X_i$ is quasi-Stein. Furthermore, as $e_{\fo}$ acts as the identity on $\OO(X_i)$, the inclusion $e_{\fo} \cdot W \hookrightarrow W$ induces an isomorphism 
    \[
    \OO(X_i) \otimes_{L_{\fo}} e_{\fo} \cdot W \xrightarrow{\sim} \OO(X_i) \otimes_{L} W.
    \]
    The same is true for $V$, and we have a commutative diagram
\[\begin{tikzcd}
	{\OO(X_i) \otimes_{L_{\fo}} e_{\fo} \cdot W} & {M_i} & {\OO(X_i) \otimes_{L_{\fo}} e_{\fo} \cdot V} \\
	{\OO(X_i) \otimes_{L} W} & {\sF(X_i)} & {\OO(X_i) \otimes_{L} V}
	\arrow[hook, from=1-1, to=1-2]
	\arrow["\sim"', from=1-1, to=2-1]
	\arrow[hook, from=1-2, to=1-3]
	\arrow["\sim"', from=1-2, to=2-2]
	\arrow["\sim"', from=1-3, to=2-3]
	\arrow[hook, from=2-1, to=2-2]
	\arrow[hook, from=2-2, to=2-3]
\end{tikzcd}\]
    where we have defined $M_i$ to be the preimage of $\sF(X_i)$ in $\OO(X_i) \otimes_{L_{\fo}} e_{\fo} \cdot V$. Suppose that
    \[
    z = \sum_{k} s_k \otimes v_k \in M_i \subset \OO(X_i) \otimes_{L_{\fo}} e_{\fo} \cdot V,
    \]
    noting that without loss of generality we may assume that the elements $s_k$ are $L_{i}$-linearly independent by choosing a $L_{i}$ basis of the $L_{i}$-span of $\{s_k\}_k$ and using the isomorphism $L_{\fo} \xrightarrow{\sim} L_i$.
    
    By assumption (4') $\OO_{X_i}$ is irreducible in $\sB_i$, and thus $\OO(X_i)$ is irreducible as an $R$-module (where $R = \sD(X_i) \rtimes H_i$ in case (A) and $R = \OO(X_i) \rtimes H_i$ in case (B)) by \cite[Prop.\ 2.44]{TAY3} as $X_i$ is quasi-Stein. Furthermore
	\[
    \End_{R}(\OO(X_i)) = L_i
	\]
 by the definition of $L_i$. Therefore, for any index $j$ we may apply the Jacobson Density Theorem to assert the existence of some $x_j \in R$ such that $x_j$ acts on the $L_i$-span of $\{ s_k \}_k$ by the projection to the basis element $s_j$. Because $H_i$ acts trivially on $e_{\fo} \cdot V$, $R$ only acts on the first factor of $\OO(X_i) \otimes_{L_{\fo}} e_{\fo} \cdot V$, and therefore $x_j(z) = s_j \otimes v_j \in \sF(X_i)$, and thus $1 \otimes v_j \in \sF(X_i)$, using again that $\OO(X_i)$ is a simple $R$-module. But then also $\OO_{X_{\fo}} \otimes v_j \subset \sF|_{X_{\fo}}$ as $\OO_{X_{\fo}}$ is irreducible in $\sB_{\fo}$ and $H$ acts trivially on $V$. Therefore, each $v_j \in e_{\fo} \cdot W$, as we can describe $e_{\fo} \cdot W$ as
 \[
    e_{\fo} \cdot W = \{v \in V \mid e_{\fo} \cdot v = v \text{ and } s \otimes v \in \sF(U) \text{ for any open } U \subset X_{\fo}, s \in \OO_{X_{\fo}}(U)\}.
 \]
In particular, $z \in \OO(X_i) \otimes_{L_{\fo}} e_{\fo} \cdot W$ and $\OO(X_i) \otimes_{L} W \hookrightarrow \sF(X_i)$ is surjective.

 Finally, to see that $W$ is free of finite rank over $L$, as $G$ acts transitively on the orbits $\fo$, and $\dim_{L_{\fo}} e_{\fo} \cdot W \leq \dim_{L_{\fo}} e_{\fo} \cdot V < \infty$, we may deduce this from Lemma \ref{lem:modoverskewisfree} once we know that the map
 \[
 W \rightarrow \prod_{\fo} e_{\fo} \cdot W
 \]
 is an isomorphism. This is injective, as the corresponding map for $V$ is an isomorphism, $V$ being free over $L$. For surjectivity, given $(w_{\fo})_{\fo}$, from the isomorphism for $V$ there is a unique $v \in V$ with $e_{\fo} \cdot v = w_{\fo}$ for all $\fo \in I$. To see that $v \in W$, suppose that $U \subset X$ is an admissible open subset, and $s = (s_{\fo})_{\fo} \in \prod_{\fo} \OO_{X_\fo}(U \cap X_{\fo}) = \OO_X(U)$. Then
 \begin{align*}
     s \otimes v &= (s_{\fo} \otimes v)_{\fo}, \\
     &= (e_{\fo} \cdot s_{\fo} \otimes v)_{\fo}, \\
     &= (s_{\fo} \otimes e_{\fo} \cdot v)_{\fo}, \\
     &= (s_{\fo} \otimes m_{\fo})_{\fo}
 \end{align*}
 which lies in $\prod_{\fo} \sF(U \cap X_{\fo}) = \sF(U)$ and therefore $v \in W$.
\end{proof}

\begin{remark}\label{rmk:simplerform}
    In case (B), evaluating on the section $1_X \in \OO(X)$ induces an isomorphism
    \[
    \Hom_{\sB}(\OO_X , \sV) \xrightarrow{\sim} \sV(X)^H
    \]
    which is natural in $\sV \in \sA$. In case (A), this further restricts to a natural isomorphism
    \[
    \Hom_{\sB}(\OO_X , \sV) \xrightarrow{\sim} \sV(X)^{\sT(X) = 0,H}
    \]
    whenever the tangent sheaf $\sT$ is generated as an $\OO$-module by global sections \cite[Lem.\ 3.1.4]{AW1}. Here $\sV(X)^{\sT(X) = 0,H}$ denotes the $H$-invariants of
    \[
    \sV(X)^{\sT(X) = 0} = \{v \in \sV(X) \mid \partial(v) = 0 \text{ for any } \partial \in \sT(X)\}.
    \]
    In particular both isomorphisms hold whenever $X$ is a disjoint union of quasi-Stein spaces.
\end{remark}

\subsection{Properties of $\Rep_L^H(G)$}\label{sect:repcat}

We can translate Lemma \ref{lem:injective} and Proposition \ref{prop:mainprop} of the previous section into properties of the categories $\Rep_L^H(G)$.

Suppose in this section that $F = \prod_i F_i$ is a product of fields, $G$ is a group which acts on $F$, and $H$ is a subgroup of $G$. By Lemma \ref{lem:autpresprincipalidem}, $G$ acts on the principal idempotents $\{e_i\}_i$ of $F$, and we assume that this action is transitive. We additionally assume that $G$ stabilises $L \coloneqq F^H$. In particular, this allows us to define a space 
\[
X = \bigsqcup_{i} \Spec(F_i)
\]
with an action of $G$ which satisfies the hypothesis of Section \ref{sect:generalities} in case (B), and from which we can recover $F$ with its action of $G$ as $F = \OO(X)$. 

The category $\Vect^G(X)$ is canonically identified with the category of $F \rtimes G$-modules $M$ for which the natural map 
\[
M \rightarrow \prod_i e_i \cdot M
\]
is an isomorphism, and $\dim_{F_i} e_i \cdot M < \infty$, which by Lemma \ref{lem:modoverskewisfree} is simply the category $\Rep_F(G)$.

\begin{cor}\label{cor:repsinjective}
    Suppose that $V \in \Rep_F(G)$. Then the natural map
    \begin{equation}\label{eqn:repsinjection}
        F \otimes_L V^H \rightarrow V
    \end{equation}
    is injective, $V^H$ is free of finite rank over $L$, and $\rank_L V^H \leq \rank_F V$.
    
    Furthermore, the following are equivalent:
    \begin{itemize}
        \item The map (\ref{eqn:repsinjection}) is an isomorphism,
        \item $F \cdot V^H = V$,
        \item $\rank_L V^H = \rank_F V$.
    \end{itemize}
\end{cor}
\begin{proof}
    All that remains to show is that if $\rank_L V^H = \rank_F V$, then (\ref{eqn:repsinjection}) is surjective. Writing $W$ for the image of (\ref{eqn:repsinjection}), it is sufficient to show that the quotient $V / W$ is free, as then $V \cong W \oplus V / W$ and hence $V / W$ must have rank $0$. For this, by Lemma \ref{lem:modoverskewisfree} it is sufficient to show that the natural map $V/W \xrightarrow{\sim} \prod_i e_i \cdot (V/W)$ is an isomorphism, which follows directly from the isomorphisms $V \xrightarrow{\sim} \prod_i e_i \cdot V$ and $W \xrightarrow{\sim} \prod_i e_i \cdot W$.
\end{proof}

When $H$ is normal in $G$, $V^H$ is closed under the action of $G$ for any $V \in \Rep_F(G)$, and therefore the invariants functor from $\Rep_F(G)$ to $\Mod_L$ admits a factorisation
\[
(-)^H \colon \Rep_F(G) \rightarrow \Rep^H_L(G).
\]
\begin{cor}\label{cor:repsfunctors}
    Suppose that $F$, $G$ and $H$ are as described at the start of this section. Then:
\begin{enumerate}
    \item
    The functor
    \[
    F \otimes_{L} - \colon \Rep^H_L(G) \rightarrow \Rep_F(G)
    \]
    is exact, monoidal, and fully faithful.
    \item
    Suppose that $H$ is normal in $G$. Then the essential image is the full subcategory of objects $V \in \Rep_F(G)$ for which $F \cdot V^H = V$, and on this full subcategory
    \[
    (-)^H \colon \Rep_F(G) \rightarrow \Rep^H_L(G)
    \]
    is a quasi-inverse to $F \otimes_{L} -$.
    \item 
    The essential image of $F \otimes_{L} -$ is closed under sub-quotients. 
    \end{enumerate}
\end{cor}

\section{Differential Operators on Lubin-Tate and Drinfeld Spaces}\label{sect:LTDRbundles}

In this section we relate the action of $\GL_n(F)$ on Drinfeld spaces with the action of the tangent sheaf. We will use this in the following sections to give a description of the categories of Lubin-Tate and Drinfeld bundles in terms of smooth semilinear representations.

\subsection{Notation}
Let $F$ be a finite extension of $\bQ_p$, and let $K$ be a complete non-archimedean field extension of the completion of the maximal unramified extension of $F$. Let $n \geq 1$, and let $D$ be the division algebra over $F$ of invariant $1/n$, with ring of integers $\OO_D$. We write $\pi$ for a uniformiser of $\OO_F$ and $\Pi$ for a uniformiser of $\OO_D$. We also consider the extensions
\[
\breve{F} \subset \breve{F}_1 \subset \breve{F}_2 \subset \cdots
\]
of the completion of the maximal unramified extension of $F$, where $\breve{F}_m$ is the compositum of $\breve{F}$ with $F_m$, the $m$th Lubin-Tate extension of $F$. Following Kohlhaase \cite{KOH}, from now on we will denote:
\begin{multicols}{2}
    \begin{itemize}
        \item $G = \GL_n(F)$,
        \item $G^0 = \ker(\nu_{\pi} \circ \det) \triangleleft G$,
        \item $G_0 = \GL_n(\OO_F)$,
        \item $G_m = 1 + \pi^m M_n(\OO_F)$, for $m \geq 1$,
    \end{itemize}
    \columnbreak
    \begin{itemize}
        \item $H = D^\times$,
        \item $H_0 = \OO_D^\times$,
        \item $H_m = 1 + \pi^m \OO_D$, for $m \geq 1$,
        \item $Z_m = G_m \times H_m$, for $m \geq 1$.
    \end{itemize}
\end{multicols}
We denote the Drinfeld tower by
\[
\Omega \leftarrow \sM \leftarrow \sM_1 \leftarrow \sM_2 \cdots
\]
which we view via base change as a tower of rigid spaces over $K$. The space $\sM_m$ and $\sM$ correspond (in the sense of local Shimura varieties \cite[\S 5.1]{RAPVIE}) to the compact open subgroups $H_m$ and $H_0$ of $H$ respectively, and correspondingly we have that each map $\sM_m \rightarrow \sM$ is a finite \'{e}tale Galois covering with Galois group $H_0 / H_m$. Let
\[
\Omega \xleftarrow{\sim} \sN \leftarrow \sN_1 \leftarrow \sN_2 \cdots 
\]
denote the subtower induced by choosing a connected component $\sN$ of $\sM$ and setting $\sN_m$ to be the preimage of $\sN$ in $\sM_m$, and write $f_m \colon \sN_m \rightarrow \Omega$ for restriction to $\sN_m$ of the $m$th covering map. Note that we are considering a cofinal subtower of that of \cite{TAY3}: what we denote by $\sM_m$ is called $\sM_{nm}$ in \cite{TAY3} (and similarly for $\sN_m$). Let
\[
(\bP^{n-1} \leftarrow \sY_{U})_{U \leq_{\text{c,o}} G} 
\]
be the Lubin-Tate tower, indexed by the compact open subgroups $U$ of $G$. We write $\sY = \sY_0 = \sY_{G_0}$, and $\sY_m = \sY_{G_m}$ for $m \geq 1$. We note that these are denoted $\underline{\sY}$ and $\underline{\sY_m}$ respectively in \cite{KOH}.

\subsection{The Action of $G_0$ and Differential Operators}\label{sect:infintesimalaction}

\begin{lemma}\label{lem:groupgeneratestangentsheaf}
Let $m \geq 1$, and let $U$ be an admissible affinoid open subset of $\sN_m$.

Then there is some $r \geq 1$ such that for any $k \geq r$,
\begin{enumerate}
    \item $G_k$ stabilises $U$,
    \item For any $g \in G_k$, $\log(g^{\OO}_U)$ converges to an element of $\sT(U)$,
    \item The image $\log(G_k)$ generates $\sT(U)$ as an $\OO(U)$-module.
\end{enumerate}
\end{lemma}

Here, for $g \in G$, $g^{\OO} \colon \OO_{\sN_m} \rightarrow g^{-1} \OO_{\sN_m}$ is the structure morphism for the action of $G$ on the $G$-equivariant sheaf $\OO_{\sN_m}$, and $g^{\OO}_U \colon \OO(U) \rightarrow \OO(g(U))$ are the sections above $U$, which is a ring automorphism of $\OO(U)$ whenever $g(U) = U$ (see \cite[\S 2.6, \S 2.8]{TAY3}).

\begin{proof}
    The action of $G^0$ on $\sN_m$ is continuous in the sense of \cite[Def.\ 3.1.8]{ARD} by \cite[Cor.\ 7.1]{TAY3}, and therefore the stabiliser of $U$ in $G^0$ is open and thus contains some $G_{s}$.
    Set $A \coloneqq \OO(U)$, and let $\rho \colon G_s \rightarrow \Aut_K(A)$ denote the action map. Choosing a formal model $\sA$ of $A$, this defines a topology on $\Aut_K(A)$ \cite[Def.\ 3.1.3]{ARD} which is independent of the choice of formal model $\sA$ of $A$ \cite[Thm.\ 3.1.5]{ARD}.
    The open subgroup $\sG_{p^{2}}(\sA) \subset \Aut_K(A)$ has the property that for any $\varphi \in \sG_{p^{2}}(\sA)$ the series $\log(\varphi)$ converges and defines an element of $\Der_{\sR}(\sA)$, where $\sR$ is the ring of integers of $K$.
    
    The continuity of the action of $G^0$ on $\sN_m$ further implies that $\rho \colon G_s \rightarrow \Aut_K(A)$ is continuous, and thus there is some $G_r \leq G_s$ with $\rho(G_r) \subset \sG_{p^{2}}(\sA)$. We therefore have that $\log(g^{\OO}_U)$ converges for any $g \in G_r$ to an element of $\Der_\sR(\sA)$, which we consider as a derivation of $A$ via the canonical inclusion $\Der_\sR(\sA) \hookrightarrow \Der_K(A)$ (noting that $A = \sA[1/p]$). 
    
    For the third claim, fix an affine chart $\bA_K^{n-1, \text{an}} \subset \bP_K^{n-1, \text{an}}$ with $\Omega \subset \bA_K^{n-1, \text{an}}$, and let $\{\Omega_j\}_{j \geq 1}$ be the quasi-Stein open covering of $\Omega$ defined by setting $\Omega_j = \Omega \cap \bD_j$, the intersection in $\bA_K^{n-1, \text{an}}$ of $\Omega$ with the affinoid $(n-1)$-dimensional ball $\bD_j = \{x \in \bA_K^n \mid |x| \leq j\}$ of centre $0 \in \bA_K^{n-1, \text{an}}$ and radius $j$. Let $\{U_j\}_{j \geq 1}$ be the induced quasi-Stein open covering of $\sN_m$, where $U_j$ is the preimage of $\Omega_j$ in $\sN_m$ under the $G_0$-equivariant finite \'{e}tale map $\sN_m \rightarrow \sN \xrightarrow{\sim} \Omega$.
    
    Because $U$ is quasi-compact, $U$ is contained in some $U_j$ for some $j \geq 1$, and the composition
    \[
    U \hookrightarrow U_j \rightarrow \Omega_j \hookrightarrow \bD_j
    \]
    is \'{e}tale and $G^0$-equivariant. The tangent sheaf $\sT(\bD_j)$ is generated over $\OO(\bD_j)$ by the derivations
    \[
    \partial_i \coloneqq \frac{d}{dx_i} \colon \OO(\bD_j) \rightarrow \OO(\bD_j),
    \]
    for coordinates $x_1, ... ,x_{n-1}$ of $\bA^{n-1, \text{an}}$. By the same argument as above, there is some $G_k$ such that $G_k$ stabilises $\bD_j$ and $\log(g^{\OO}_{\bD_j})$ converges for any $g \in G_k$ which (after potentially going to a smaller congruence subgroup) we may assume is the same $G_k$ as above. From the formula for the logarithm, 
    \[
    \pi^k \partial_i = \log(g^{\OO}_{i,k, \bD_j})
    \]
    where $g_{i,k} \in G_k$ is the elementary matrix in $G_k$ which acts by the M\"{o}bius transformation which sends $x_i \mapsto x_i + \pi^k$, and fixes each other $x_j$. 

By \cite[Lem.\ 2.12]{TAY3}, because the morphism $U \rightarrow \bD_j$ is \'{e}tale, for each $i$ there is a unique derivation $\partial'_i$ of $\OO(U)$ which restricts to $\partial_i$ on the image of $\OO(\bD_j)$ in $\OO(U)$, and these generate $\sT(U)$ as an $\OO(U)$-module because the derivations $\partial_i$ generate $\sT(\bD_j)$ as an $\OO(\bD_j)$-module. Because the morphism $\OO(\bD_j) \rightarrow \OO(U)$ is continuous (being a morphism of affinoid algebras), $\log(g^{\OO}_{i,k,U})$ restricts to $\log(g^{\OO}_{i,k,\bD_j})$ on the image of $\OO(\bD_j)$ in $\OO(U)$. In particular, from the uniqueness of $\partial'_i$ described above, $\partial'_i$ is explicitly given by
\[
\partial'_i = \pi^{-k} \log(g^{\OO}_{i,k,U}),
\]
and therefore $\log(G_k)$ generates $\sT(U)$ as an $\OO(U)$-module.
\end{proof}

\begin{cor}\label{cor:irred}
    For any connected component $X$ of $\sN_m$, $\OO_{X}$ is irreducible as a $G_m\text{-}\OO_{X}$-module.
\end{cor}

\begin{remark}
    For the above to make sense, we note that any connected component of $\sN_m$ is stabilised by $G_m$. Indeed, when $K$ contains $\breve{F}_m$, $\pi_0(\sN_m)$ is canonically identified with $\OO_F / (1 + \pi^m \OO_F)$, and the action of $G_0$ on $\pi_0(\sN_m)$ corresponds to the action of $g \in G_0$ on $\OO_F / (1 + \pi^m \OO_F)$ by left multiplication by $\det(g)$ \cite[Thm.\ 7.3]{TAY3}. In particular, $G_m$ acts on $\pi_0(\sN_m)$ trivially. For a general $K$, there is a $G_0$-equivariant surjection $\pi_0(\sN_{m, K\breve{F}_m}) \twoheadrightarrow \pi_0(\sN_m)$, and therefore $G_m$ also acts on $\pi_0(\sN_m)$ trivially.
\end{remark}

\begin{proof}
    To prove statement (1), suppose that $\sV$ is a non-trivial $G_m \text{-}\OO_{X}$-submodule of $\OO_{X}$. We show that $\sV$ is in fact a $\sD_{X}$-submodule of $\OO_{X}$, from which it follows that $\sV = \OO_{X}$ as $\OO_{X}$ is irreducible as a $\sD_{X}$-module because $X$ is connected \cite[Lem.\ 2.37]{TAY3}.

    To this end, let $U \subset X$ be a connected admissible open affinoid subset. By Lemma \ref{lem:groupgeneratestangentsheaf}, we can find some $G_k \subset G_m$ such that $G_k$ stabilises $U$ and $\log(G_k)$ generates $\sT(U)$ as an $\OO(U)$-module, and thus it is sufficient to show that for any $g \in G_k$ and $v \in \sV(U)$, $\log(g_U^\OO)(v) \in \sV(U)$. Because $\OO(U)$ is a noetherian Banach algebra, $\sV(U)$ is closed in the subspace topology \cite[Prop.\ 2.1]{STALG}, and therefore
    \[
    \log(g_U^\OO)(v) = \sum_{k = 1}^{\infty} \frac{(-1)^{k+1}}{k} (g_U^\OO - 1)^k(v) \in \sV(U),
    \]
    as each $(g_U^\OO - 1)^k(v) \in \sV(U)$ by the assumption that $\sV \subset \OO_X$ is a $G_m \text{-}\OO_{X}$-submodule.
\end{proof}

\begin{remark}
    For our applications, we will only actually need the consequence that $\OO_{\sN_m}$ is irreducible in $\Coh^{G_m}(\sN_m)$, not in the larger category $\Mod(G_m \text{-} \OO_{\sN_m})$ as we have proven here.
\end{remark}

\section{$G_0$-Equivariant Lubin-Tate Bundles on $\Omega$}\label{sect:DRside0}

In this section, we show that the category of $G_0$-equivariant Lubin-Tate bundles on $\Omega$ introduced by Kohlhaase is equivalent to a certain category of smooth semilinear representations of $G_0 \times H_0$.

\begin{defn}
    For $m \geq 1$, we set $K_m = c(\sN_m)$.
\end{defn}

Here $c(X)$ is the $K$-algebra of global sections of the sheaf of constant functions \cite[Def.\ 3.1]{TAY3}.

\begin{defn}
    An object $\sV$ of $\Vect^{G_0}(\Omega)$ or $\VectCon^{G_0}(\Omega)$ is \emph{Lubin-Tate} if the natural map
    \begin{equation}\label{eqn:caninjectionLT}
        \OO_{\sN_m} \otimes_{K_m} (f_m^*\sV)(\sN_m)^{G_m} \rightarrow f_m^*\sV
    \end{equation}
    is an isomorphism for some $m \geq 1$, in which case we say that $\sV$ is \emph{Lubin-Tate of level $m$}. We let
    \begin{align*}
    \VectCon^{G_0}_{\LT}(\Omega) &= \bigcup_{m \geq 1} \VectCon^{G_0}_{\LT,m}(\Omega), \\
    \Vect^{G_0}_{\LT}(\Omega) &= \bigcup_{m \geq 1} \Vect^{G_0}_{\LT,m}(\Omega),
\end{align*}
denote the corresponding full subcategories of $G_0$-equivariant Lubin-Tate bundles on $\Omega$.
\end{defn}

\begin{remark}
In particular, by \cite[Lem.\ 3.11]{KOH} and the description of $K_m$ below, the category $\Vect^{G_0}_{\LT}(\Omega)$ we have defined agrees with Kohlhaase's category of $G_0$-equivariant Lubin-Tate bundles.
\end{remark}

The product of fields $K_m$ is described by
\[
K_m = c(\sN_m) = c(\sN_m)^{G_m} = \OO(\sN_m)^{G_m} = K \otimes_{\breve{F}} c(\sN_{m,\breve{F}}) =  K \otimes_{\breve{F}} \breve{F}_m,
\]
by a result of Kohlhaase \cite[Thm.\ 2.8(ii)]{KOH}. The group $G_0 \times H_0$ acts on $K_m$ through the right factor, via the composition of the homomorphism
\[
G_0 \times H_0 \rightarrow \OO_F / (1 + \pi^m \OO_F), \qquad (g,h) \mapsto \det(g) \Nrd(h)^{-1},
\]
with the Lubin-Tate isomorphism
\[
\OO_F / (1 + \pi^m \OO_F) \xrightarrow{\sim} \Gal(\breve{F}_m / \breve{F}).
\]
With this action, we can consider the semilinear representation category $\Rep_{K_m}^{Z_m}(G_0 \times H_0)$.

We can take the triple $(X, G, H)$ of Section \ref{sect:generalities} to be $(\sN_m, G_0 \times H_0 / H_m, G_m)$, which in either case (A) or (B) has $L = K_m$ by the above discussion. This is preserved by the action of $G_0 \times H_0$, and in either case the triple satisfies the assumptions (1), (2), (3) and (4) of Section \ref{sect:generalities} by Corollary \ref{cor:irred}. In particular, we may compose the functors $(\OO_{\sN_m} \otimes_{K_m} -)$ of Section \ref{sect:generalities} with the equivalence of \cite[Prop.\ 2.53]{TAY3} (which also holds for vector bundles without connection, using the same proof), to obtain functors
\begin{equation}\label{eqn:firstpairoffunctors}
\begin{aligned}
    (\OO_{\sN_m} \otimes_{K_m} -)^{H_0} &\colon \Rep^{Z_m}_{K_m}(G_0 \times H_0) \rightarrow \VectCon^{H_0 / H_m \times G_0}(\sN_m) \xrightarrow{\sim} \VectCon^{G_0}(\Omega), \\
    (\OO_{\sN_m} \otimes_{K_m} -)^{H_0} &\colon \Rep^{Z_m}_{K_m}(G_0 \times H_0) \rightarrow \Vect^{H_0 / H_m \times G_0}(\sN_m) \xrightarrow{\sim} \Vect^{G_0}(\Omega).
\end{aligned}
\end{equation}

\begin{prop}\label{prop:levelmLTbundles}
    For $m \geq 1$, the map (\ref{eqn:caninjectionLT}) is always injective, and the composite functors (\ref{eqn:firstpairoffunctors}) are exact, monoidal, and fully faithful. The essential images are $\VectCon^{G_0}_{\LT,m}(\Omega)$ and $\Vect^{G_0}_{\LT,m}(\Omega)$ respectively, both of which are closed under sub-quotients. In particular, the canonical restriction map
    \[
    \VectCon^{G_0}_{\LT,m}(\Omega) \rightarrow \Vect^{G_0}_{\LT,m}(\Omega)
    \]
    is an equivalence of categories.
\end{prop}

\begin{proof}
    The map (\ref{eqn:caninjectionLT}) is the canonical injection of Lemma \ref{lem:injective}, using Remark \ref{rmk:simplerform} and the fact that $\sN_m$ is quasi-Stein. Then the result follows directly from Proposition \ref{prop:mainprop}, using that $\sN_m$ is quasi-Stein of finite dimension to give that the image is closed under sub-quotients.
\end{proof}
For $m' \geq m$, there is a fully faithful inclusion functor
\[
K_{m'} \otimes_{K_m} - \colon \Rep^{Z_m}_{K_m}(G_0 \times H_0) \rightarrow \Rep^{Z_{m'}}_{K_{m'}}(G_0 \times H_0),
\]
and this is compatible with inclusion of Lubin-Tate bundles of levels $m$ and $m'$ in the sense that
\[\begin{tikzcd}[ampersand replacement=\&]
	{\Rep^{Z_{m'}}_{K_{m'}}(G_0 \times H_0)} \& {\VectCon^{G_0}_{\LT,m'}(\Omega)} \& {\Vect^{G_0}_{\LT,m'}(\Omega)} \\
	{\Rep^{Z_m}_{K_m}(G_0 \times H_0) } \& {\VectCon^{G_0}_{\LT,m}(\Omega)} \& {\Vect^{G_0}_{\LT,m}(\Omega)}
	\arrow["\sim", from=1-1, to=1-2]
	\arrow["\sim", from=1-2, to=1-3]
	\arrow[hook, from=2-1, to=1-1]
	\arrow["\sim", from=2-1, to=2-2]
	\arrow[hook, from=2-2, to=1-2]
	\arrow["\sim", from=2-2, to=2-3]
	\arrow[hook, from=2-3, to=1-3]
\end{tikzcd}\]
commutes. In particular, we see that if $\sV \in \Vect^{G_0}(\Omega)$ is Lubin-Tate of level $m$, it is Lubin-Tate of level $m'$ for all $m' \geq m$ (and similarly for $\VectCon^{G_0}(\Omega)$). We can use this to describe the whole categories $\Vect^{G_0}_{\LT}(\Omega)$ and $\VectCon^{G_0}_{\LT}(\Omega)$ in terms of smooth semilinear representations of $G_0 \times H_0$.

\begin{defn}
    We write $\breve{F}_{\infty}$ for the union of the fields $\breve{F}_m$, and set
    \[
    K_{\infty} = \varinjlim_{m \geq 1} K_m = K \otimes_{\breve{F}} \breve{F}_{\infty},
    \]
    which we consider with its natural action of $G_0 \times H_0$ through the factor $\breve{F}_{\infty}$ as above.
\end{defn}

For each $m \geq 1$, there is a canonical inclusion functor
\[
K_{\infty} \otimes_{K_m} - \colon \Rep^{Z_m}_{K_m}(G_0 \times H_0) \rightarrow \Rep_{K_{\infty}}(G_0 \times H_0),
\]
the essential image of which we denote by $\Rep^m_{K_{\infty}}(G_0 \times H_0)$.
\begin{lemma}\label{lem:unionofcatsK}
    The category $\Rep^m_{K_{\infty}}(G_0 \times H_0)$ is intrinsically described as
    \[
    \Rep^m_{K_{\infty}}(G_0 \times H_0) = \{V \in \Rep_{K_{\infty}}(G_0 \times H_0) \mid K_{\infty} \cdot V^{Z_m} = V \},
    \]
    and is closed under sub-quotients. Furthermore, on this full subcategory 
    \[
    (-)^{Z_m} \colon \Rep^m_{K_{\infty}}(G_0 \times H_0) \rightarrow  \Rep^{Z_m}_{K_m}(G_0 \times H_0)
    \]
    is a quasi-inverse to the exact, monoidal, fully faithful functor $K_{\infty} \otimes_{K_m} -$, and in $\Rep_{K_{\infty}}(G_0 \times H_0)$
    \[
    \Rep^{\sm}_{K_{\infty}}(G_0 \times H_0) = \bigcup_{m \geq 1} \Rep^{m}_{K_{\infty}}(G_0 \times H_0).
    \]
\end{lemma}

\begin{proof}
    This follows directly from Corollary \ref{cor:repsfunctors} as $K_m = K_{\infty}^{Z_m}$. The final claim follows from the description of $\Rep^m_{K_{\infty}}(G_0 \times H_0)$, and the fact that the subgroups $Z_m$ form a neighbourhood basis of the identity in $G_0 \times H_0$.
\end{proof}

We therefore have equivalences
\[
\Phi_{\LT,m}^0 \coloneqq (\OO_{\sN_m} \otimes_{K_m} (-)^{Z_m})^{H_0} \colon \Rep^{m}_{K_{\infty}}(G_0 \times H_0) \xrightarrow{\sim} \VectCon^{G_0}_{\LT,m}(\Omega)
\]
which are compatible by the above discussion and we may therefore take the direct limit
\[
\Phi^0_{\LT} \coloneqq \varinjlim_{m \geq 1} \Phi_{\LT, m}^0 \colon \Rep^{\sm}_{K_{\infty}}(G_0 \times H_0) \rightarrow \VectCon^{G_0}_{\LT}(\Omega).
\]
Because $\Phi_{\LT,m}^0$ is an equivalence at each level, we have the following.

\begin{thm}\label{thm:mainthmLT0}
    The functors
    \[
    \Rep^{\sm}_{K_{\infty}}(G_0 \times H_0) \xrightarrow{\Phi^0_{\LT}} \VectCon^{G_0}_{\LT}(\Omega) \rightarrow \Vect^{G_0}_{\LT}(\Omega)
    \]
    are equivalences of categories.
\end{thm}

\subsection{$H_0$-Equivariant Drinfeld Bundles on $\sY$}\label{sect:LTside0}

Kohlhaase defined equivalences,
\[
\bD_{\LT} \colon \Vect^{G_0}_{\LT}(\Omega) \xleftrightarrow{\sim} \Vect^{H_0}_{\Dr}(\sY) : \! \bD_{\Dr},
\]
where $\Vect^{H_0}_{\Dr}(\sY)$ is an analogously defined category of $H_0$-equivariant \emph{Drinfeld bundles} on $\sY$. 

We can similarly define, for each level $m \geq 1$,
\[
(\OO_{\sY_m} \otimes_{K_m} -)^{G_0} \colon \Rep^{Z_m}_{K_m}(G_0 \times H_0) \rightarrow \Vect^{G_0/G_m \times H_0}(\sY_m) \xrightarrow{\sim} \Vect^{H_0}(\sY)
\]
with essential image denoted $\Vect^{H_0}_{\Dr,m}(\sY)$, and have functors
\[
\Phi_{\Dr,m}^0 \coloneqq (\OO_{\sY_m} \otimes_{K_m} (-)^{Z_m})^{G_0} \colon \Rep^{m}_{K_{\infty}}(G_0 \times H_0) \rightarrow \Vect^{H_0}_{\Dr,m}(\sY)
\]
for which we can take the direct limit
\[
\Phi^0_{\Dr} \coloneqq \varinjlim_{m \geq 1} \Phi_{\Dr, m}^0 \colon \Rep^{\sm}_{K_{\infty}}(G_0 \times H_0) \rightarrow \Vect^{H_0}_{\Dr}(\sY).
\]

We can deduce that $\Phi^0_{\Dr}$ and each $\Phi_{\Dr,m}^0$ are equivalences from the following compatibility result.

\begin{prop}\label{prop:compatiblewithequiv}
For any $m \geq 1$, the composition
\[
\Rep^{Z_m}_{K_m}(G_0 \times H_0) \xrightarrow{\Phi^0_{\LT,m}} \Vect^{G_0}_{\LT}(\Omega) \xrightarrow{\bD_{\LT}}  \Vect^{H_0}_{\Dr,m}(\sY)
\]
is naturally isomorphic to $\Phi_{\Dr,m}^0$ and $\Phi_{\Dr,m}^0$ is an equivalence. In particular, the composition 
\[
\Rep_{K_{\infty}}^{\sm}(G_0 \times H_0) \xrightarrow{\Phi^0_{\LT}} \Vect^{G_0}_{\LT}(\Omega) \xrightarrow{\bD_{\LT}} \Vect^{H_0}_{\Dr}(\sY)
\]
is naturally isomorphic to $\Phi^0_{\Dr}$ and $\Phi^0_{\Dr}$ is an equivalence.
\end{prop}

\begin{proof}
For $m \geq 1$, suppose that $\sV \in \Vect^{G_0}_{\LT}(\Omega)$. As all spaces are quasi-Stein, it is sufficient to give the isomorphism on global sections. Writing $M = \Gamma(\Omega, \sV)$, then in the notation of \cite{KOH},
    \begin{align*}
        \Gamma(\Omega, \bD_{\LT}(\sV))
        &= (C_m \otimes_{A_0} M)^{G_0}, \\
        &\equiv (B_m \hat{\otimes}_{K_m} A_m \otimes_{A_0} M)^{G_0},\\
        &\equiv (B_m \hat{\otimes}_{K_m} (A_m \otimes_{A_0} M)^{G_m})^{G_0 / G_m}, \\
        &\equiv (B_m \hat{\otimes}_{K_m} (A_m \otimes_{K_m} V)^{G_m})^{G_0 / G_m}.
    \end{align*}
    Here for the third equality we use \cite[Lem.\ 3.3]{KOH} and the fact that $G_m$ acts trivially on $B_m$. In particular, if $\sV = \Phi^0_{\LT,m}(V)$, then $A_m \otimes_{A_0} M = A_m \otimes_{K_m} V$ and
    \begin{align*}
        B_m \hat{\otimes}_{K_m} (A_m \otimes_{A_0} M)^{G_m}
        &= B_m \hat{\otimes}_{K_m} (A_m \otimes_{K_m} V)^{G_m}, \\
        &= B_m \hat{\otimes}_{K_m} (K_m \otimes_{K_m} V), \\
        &\equiv B_m \otimes_{K_m} V,
    \end{align*}
    using that $G_m$ acts trivially on $V$, and thus there is a natural isomorphism
    \begin{align*}
        \Gamma(\Omega, \bD_{\LT}(\Phi^0_{\LT,m}(V))) &\equiv (B_m \otimes_{K_m} V)^{G_0 / G_m}, \\
        &= \Gamma(\Omega, \Phi^0_{\Dr,m}(V)). \qedhere
    \end{align*}
\end{proof}

\begin{remark}
    One can define analogous categories $\VectCon^{G_0}_{\LT}(\Omega)$ and $\VectCon^{H_0}_{\Dr}(\sY)$, and by defining appropriate $\sD$-module structures extend Kohlhaase's equivalences to equivalences
\[
\bD_{\LT}^{\nabla} \colon \VectCon^{G_0}_{\LT}(\Omega) \xleftrightarrow{\sim} \VectCon^{H_0}_{\Dr}(\sY) : \! \bD_{\Dr}^{\nabla},
\]
which are extensions in the sense that there are natural forgetful maps from these categories to those of Kohlhaase for which the diagram
\[\begin{tikzcd}[column sep=6em]
	{\VectCon^{G_0}_{\LT}(\Omega)} & {\VectCon^{H_0}_{\Dr}(\sY)} \\
	{\Vect^{G_0}_{\LT}(\Omega)} & {\Vect^{H_0}_{\Dr}(\sY)}
	\arrow["{\bD_{\LT}^{\nabla}}", curve={height=-6pt}, from=1-1, to=1-2]
	\arrow["\sim"{description}, draw=none, from=1-1, to=1-2]
	\arrow[from=1-1, to=2-1]
	\arrow["{\bD_{\Dr}^{\nabla}}", curve={height=-6pt}, from=1-2, to=1-1]
	\arrow[from=1-2, to=2-2]
	\arrow["{\bD_{\LT}}", curve={height=-6pt}, from=2-1, to=2-2]
	\arrow["\sim"{description}, draw=none, from=2-1, to=2-2]
	\arrow["{\bD_{\Dr}}", curve={height=-6pt}, from=2-2, to=2-1]
\end{tikzcd}\]
commutes. With these equivalences, one can further check that the the natural isomorphisms of Proposition \ref{prop:compatiblewithequiv} respect the $\sD$-module structure, and hence that Proposition \ref{prop:compatiblewithequiv} continues to hold with $\Vect$ replaced by $\VectCon$ everywhere.
\end{remark}

\subsection{Applications}

From now on we focus on the Drinfeld side, but in light of Section \ref{sect:LTside0} all the following results can be translated to the Lubin-Tate side. In this section we apply the results of Section \ref{sect:DRside0} and Section \ref{sect:repcat} to understand the categories $\VectCon^{G_0}_{\LT}(\Omega)$ and $\Vect^{G_0}_{\LT}(\Omega)$. First, we have the following immediate consequence.

\begin{cor}\label{cor:LT0SS}
    The categories $\VectCon^{G_0}_{\LT}(\Omega)$ and $\Vect^{G_0}_{\LT}(\Omega)$ are semisimple.
\end{cor}

\begin{proof}
Each category $\Rep_{K_m}^{Z_m}(G_0 \times H_0)$ is semisimple, by \cite[Cor.\ 0.2]{MONT2}. Each inclusion
\[
K_{\infty} \otimes_{K_m} - \colon \Rep_{K_m}^{Z_m}(G_0 \times H_0) \xrightarrow{\sim} \Rep_{K_{\infty}}^m(G_0 \times H_0) \hookrightarrow  \Rep^{\sm}_{K_{\infty}}(G_0 \times H_0)
\]
is closed under sub-quotients by Lemma \ref{lem:unionofcatsK} and thus $\Rep^{\sm}_{K_{\infty}}(G_0 \times H_0)$ is semisimple.
\end{proof}

\begin{remark}
    Both $\VectCon^{G_0}_{\LT}(\Omega)$ and $\Vect^{G_0}_{\LT}(\Omega)$ are closed under sub-quotients in the ambient categories $\VectCon^{G_0}(\Omega)$ and $\Vect^{G_0}(\Omega)$ by Proposition \ref{prop:mainprop}. Therefore any semisimple decomposition is also semisimple in the larger category.
\end{remark}

In light of Theorem \ref{thm:mainthmLT0}, to understand $\VectCon^{G_0}_{\LT}(\Omega)$ and $\Vect^{G_0}_{\LT}(\Omega)$ we are reduced to understanding the category $\Rep^{\sm}_{K_{\infty}}(G_0 \times H_0)$. One source of objects is the category $\Rep^{\sm}_K(G_0 \times H_0)$.

\begin{defn}
    We denote by $K_{\infty} \otimes_K -$ the canonical base change functor
\[
K_{\infty} \otimes_K - \colon \Rep^{\sm}_K(G_0 \times H_0) \rightarrow \Rep^{\sm}_{K_{\infty}}(G_0 \times H_0)
\]
from the category of smooth (linear) representations of $G_0 \times H_0$ over $K$, and set
\[
\Psi^0_{G \times H}(-) = \Phi^0_{\LT}(K_{\infty} \otimes_K -) \colon \Rep_K^{\sm}(G_0 \times H_0) \rightarrow \VectCon^{G_0}_{\LT}(\Omega)
\]
for the composition of $K_{\infty} \otimes_K -$ with the equivalence $\Phi^0_{\LT}$.
\end{defn}

\begin{remark}
By abuse of notation we also write $\Psi^0_{G \times H}$ for the composition of $\Psi^0_{G \times H}$ with the forgetful map to $\Vect^{G_0}_{\LT}(\Omega)$. Additionally, if $V \in \Rep_{K}^{Z_m}(G_0 \times H_0)$, then as the diagram 
\[\begin{tikzcd}
	{\Rep^{Z_m}_{K}(G_0 \times H_0)} & {\Rep^{Z_m}_{K_m}(G_0 \times H_0)} \\
	{\Rep^{\sm}_{K}(G_0 \times H_0)} & {\Rep^{\sm}_{K_\infty}(G_0 \times H_0)}
	\arrow["{K_m \otimes_K -}", from=1-1, to=1-2]
	\arrow[hook, from=1-1, to=2-1]
	\arrow["{K_{\infty} \otimes_{K_m} -}", from=1-2, to=2-2]
	\arrow["{K_{\infty} \otimes_K -}"', from=2-1, to=2-2]
\end{tikzcd}\]
commutes, the associated Lubin-Tate bundle is explicitly described by
\begin{equation}\label{eqn:ltexplicit}
    \begin{aligned}
\Psi^0_{G \times H}(V) &= \Phi^0_{\LT,m}(K_m \otimes_K V), \\
&= (\OO_{\sN_m} \otimes_{K_m} (K_m \otimes_K V))^{H_0/H_m}, \\
&= (\OO_{\sN_m} \otimes_K V)^{H_0/H_m}.
    \end{aligned}
\end{equation}
\end{remark}

\begin{defn}
    We define
    \begin{align*}
    \Psi_{G}^0 &\colon \Rep_K^{\sm}(G_0) \rightarrow \VectCon_{\LT}^{G_0}(\Omega),\\
    \Psi_{H}^0 &\colon \Rep_K^{\sm}(H_0) \rightarrow \VectCon_{\LT}^{G_0}(\Omega),
    \end{align*}
    to each be the post-composition of $\Psi^0_{G \times H}$ with the respective inflation functor
    \begin{align*}
    \Rep_K^{\sm}(G_0) \rightarrow \Rep_K^{\sm}(G_0 \times H_0),\\
    \Rep_K^{\sm}(H_0) \rightarrow \Rep_K^{\sm}(G_0 \times H_0),
    \end{align*}
    induced by the corresponding projection from $G_0 \times H_0$ to $G_0$ or $H_0$.
\end{defn}

Similarly to above, when composed with the forgetful equivalence, these give functors to $\Vect_{\LT}^{G_0}(\Omega)$ which we also denote by $\Psi_G^0$ and $\Psi_H^0$. Kohlhaase also considered functors from $\Rep_K^{\sm}(G_0)$ and $\Rep_K^{\sm}(H_0)$ to $\Vect_{\LT}^{G_0}(\Omega)$ \cite[Thm.\ 3.8]{KOH}. The following shows that these coincide with $\Psi_{G}^0$ and $\Psi_{H}^0$.

\begin{lemma}\label{lem:explicitformofGH0}
    For $V \in \Rep_K^{\sm}(G_0)$ and $W \in \Rep_K^{\sm}(H_0)$ there are natural isomorphisms,
    \begin{align*}
        \Psi_{G}^0(V) &\cong \OO_{\Omega} \otimes_K V, \\
        \Psi_{H}^0(W) &\cong \bD_{\LT}(\OO_{\sY} \otimes_K W)
    \end{align*}
    in $\Vect_{\LT}^{G_0}(\Omega)$.
\end{lemma}

\begin{proof}
    From the explicit description (\ref{eqn:ltexplicit}) above, if $G_m$ acts trivially on $V$, then
    \begin{align*}
    \Psi_{G}^0(V)
    &= (\OO_{\sN_m} \otimes_K V)^{H_0/H_m}, \\
    &= \OO_{\sN_m}^{H_0/H_m} \otimes_K V, \\
    &= \OO_{\Omega} \otimes_K V
    \end{align*}
    because $H_0$ acts trivially on $V$. The claim for $\Psi_{H}^0(W)$ follows similarly using Proposition \ref{prop:compatiblewithequiv}.
\end{proof}

\begin{lemma}
    The functors $\Psi_{G}^0$ and $\Psi_{H}^0$ are fully faithful and preserve irreducibility.
\end{lemma}
\begin{proof}
    The follows directly from Corollary \ref{cor:repsfunctors} with triples $(K_{\infty}, G_0 \times H_0,H_0)$ and $(K_{\infty}, G_0 \times H_0,G_0)$ respectively, as $K_{\infty}^{G_0} = K = K_{\infty}^{H_0}$.
\end{proof}

Using $\Psi_{G}^0$ and $\Psi_{H}^0$, we can describe the category of $G_0$-equivariant Lubin-Tate bundles more concretely when $n = 1$. 

\begin{lemma}\label{lem:GH0areequiv}
    When $n=1$, each functor of the diagram
    \[
    \Rep_K^{\sm}(G_0) \xrightarrow{K_{\infty} \otimes_K -} \Rep_{K_{\infty}}^{\sm}(G_0 \times H_0) \xrightarrow{\Phi_{\LT}^0} \VectCon^{G_0}_{\LT}(\Omega) \xrightarrow{\sim} \Vect^{G_0}_{\LT}(\Omega),
    \]
    is an equivalence, and the same holds for $\Rep_K^{\sm}(G_0)$ replaced with $\Rep_K^{\sm}(H_0)$.
    
    In particular, when $n = 1$ both $\Psi_G^0$ and $\Psi_H^0$ are equivalences of categories.
\end{lemma}

\begin{proof}
    The functor $K_{\infty} \otimes_K -$ is the direct limit of the functors
    \[
    K_m \otimes_K - \colon \Rep_K(G_0 / G_m) \rightarrow \Rep_{K_m}(G_0 / G_m \times H_0/H_m),
    \]
    which are equivalences with inverse $(-)^{H_0 / H_m}$ when $n = 1$ by Galois descent \cite[Prop.\ 2.53]{TAY3}, as $H_0 / H_m$ is the Galois group of $K_m / K$. The same holds similarly for $G_0$ and $H_0$ swapped.
\end{proof}

However for $n \geq 2$, the essential images of $\Psi_G^0$ and $\Psi_H^0$ will not coincide in general, as was noted by Kohlhaase \cite[Rem.\ 3.9]{KOH}. From the work we have done above, we can now give a precise description of the intersection of the essential images of $\Psi_G^0$ and $\Psi_H^0$.

\begin{thm}\label{thm:intersectEI}
The intersection of the essential images of $\Psi_G^0$ and $\Psi_H^0$ is the full subcategory consisting of objects 
\[
\Psi_G^0(V) = \OO_{\Omega} \otimes V
\]
for $V \in \Rep_K^{\sm}(\OO_F^\times)$, viewed as an object of $\Rep_K^{\sm}(G_0)$ via inflation along $\det \colon G_0 \rightarrow \OO_F^\times$.
\end{thm}

\begin{remark}
    When $(n,p) \neq (2,2)$ the determinant $\det \colon G_0 \rightarrow \OO_F^{\times}$ is an abelianization of $G_0$ \cite{LIT}, and therefore when $K$ is a splitting field for $\Rep_K^{\sm}(\OO_F^\times)$ (e.g.\ if $K$ is algebraically closed), then the intersection of Proposition \ref{thm:intersectEI} is simply the full subcategory of direct sums of $\OO_{\Omega} \otimes \chi$, for $\chi$ a smooth character of $G_0$ over $K$.
\end{remark}

\begin{proof}
    Fix an object $V \in \Rep_K^{\sm}(G_0)$, and let $m \geq 1$ be such that $V^{G_m} = V$. We take $L$ to be a finite extension of $K$ which is a splitting field for the group $A \coloneqq \OO_F^{\times} / (1+\pi^m \OO_F)$, and set $L_m \coloneqq L \otimes_K K_m$ (which has an action of $G_0 \times H_0$ which is trivial on the first component). We first note that by Theorem \ref{thm:mainthmLT0} and Lemma \ref{lem:unionofcatsK}, the following statements are equivalent:
    \begin{itemize}
        \item $\Psi_G^0(V)$ is in the essential image of $\Psi_H^0 \colon \Rep_K^{\sm}(H_0) \rightarrow \Vect^{G_0}_{\LT}(\Omega)$,
        \item $K_{\infty} \otimes_K V$ is in the essential image of $K_{\infty} \otimes_K - \colon \Rep_K^{\sm}(H_0) \rightarrow \Rep_{K_{\infty}}^{\sm}(G_0 \times H_0)$,
        \item $K_m \otimes_K V$ is in the essential image of $K_{m} \otimes_K - \colon \Rep_K^{H_m}(H_0) \rightarrow \Rep_{K_m}^{Z_m}(G_0 \times H_0)$.
    \end{itemize}
    We are therefore reduced to showing that this final condition is equivalent to $V$ being inflated from $\Rep_K^{\sm}(\OO_F^\times)$. We first consider $V_L$. Taking the triple $(F,G,H)$ in Corollary \ref{cor:repsinjective} and Corollary \ref{cor:repsfunctors} to be $(L_m, G_0 \times H_0, G_0)$, this essential image of the functor
    \[
    L_m \otimes_L - \colon \Rep_L^{H_m}(H_0) \rightarrow \Rep^{Z_m}_{L_{m}}(G_0 \times H_0),
    \]
    consists of exactly those $U$ for which the inequality
    \[
    \dim_L U^{G_0/G_m} \leq \dim_{L_m}(U)
    \]
    is an equality. Taking $U = L_m \otimes_L V_L$ this inequality becomes
    \begin{equation}\label{eqn:essiminequality}
    \dim_L(L_m \otimes_L V_L)^{G_0/ G_m} \leq \dim_K V_L.
    \end{equation}
    By the normal basis theorem \cite[Thm.\ 4.2(c)]{CHR} we have that $L_m \cong L[A]$ as $L[A]$-modules because $L_m/L$ is Galois with Galois group $A$. Writing $V_L$ as the direct sum of irreducible representations $V_i$, the inequality (\ref{eqn:essiminequality}) is an equality if and only it is for each $V_i$, and, by our assumption that $L$ is a splitting field of $A$, (\ref{eqn:essiminequality}) is an equality for a fixed $V_i$ if and only if some twist of $V_i$ by an inflated character of $A$ is trivial, or equivalently that $V_i$ is an inflated character of $A$.

    In summary, $L_m \otimes_L V_L$ is in the essential image of the functor $L_m \otimes_L -$ above if and only if $V_L$ is inflated from a smooth representation of $\OO_F^\times$. 
    In particular, if $K_m \otimes_K V \cong K_m \otimes_K W$ for some $W \in \Rep_K^{H_m}(H_0)$, then $V_L \cong W_L$ and $V_L$ is in the essential image of $L_m \otimes_L -$, so $V_L$ and hence $V$ is inflated from a smooth representation of $\OO_F^\times$. Conversely, suppose that $V$ is inflated from a smooth representation of $\OO_F^\times$. Then $V_L$ is too and hence $L_m \otimes_L V_L$ is in the essential image of $L_m \otimes_L -$. We claim that $K_m \otimes V \cong K_m \otimes W$ for $W \coloneqq (K_m \otimes_K V)^{G_0 / G_m}$. Indeed, the natural map
    \[
    K_m \otimes_K W \rightarrow V
    \]
    is an isomorphism, because after applying $L \otimes_K -$ it becomes the map
    \[
        L_m \otimes_L (L_m \otimes_L V_L)^{G_0 / G_m} \rightarrow L_m \otimes_L V_L,
    \]
    which is an isomorphism by Corollary \ref{cor:repsinjective} and Corollary \ref{cor:repsfunctors}.
\end{proof}

\begin{remark}\label{rem:commchar0}
In fact, using the results of \cite{TAYSchur} it is possible to show that $\Psi^0_G(\chi) \cong \Psi_H^0(\chi)$ for any smooth character $\chi$ of $\OO_F^\times$ viewed as a character of $G_0$ and $H_0$ by inflation along $\det$ and $\Nrd$. Later, we will show that the analogous statement for the groups $G$ and $H$ holds (cf.\ Theorem \ref{thm:twists}). For any such $\chi$ as above, $\chi \otimes 1$ and $1 \otimes \chi$ are non-isomorphic objects of $\Rep^{\sm}_K(G_0 \times H_0)$ which become isomorphic in $\Rep^{\sm}_{K_{\infty}}(G_0 \times H_0)$. In particular, unlike both $\Psi^0_{G}$ and $\Psi^0_{H}$, the functor $\Psi^0_{G \times H}$ is not fully faithful.

Again, using the results of \cite{TAYSchur} it is possible to show that $\Psi^0_{G \times H}$ is not essentially surjective in general. However, one can at least show that each object of $\Vect^{G_0}_{\LT}(\Omega)$ is a sub-object of an object in its essential image. On the other hand, we will show that the analogous functor $\Psi_{G \times H}$ for the groups $G$ and $H$ \emph{is} in fact essentially surjective (cf.\ Theorem \ref{thm:finalmainthmLT}).
\end{remark}

\section{$G$-Equivariant Lubin-Tate Bundles on $\Omega$}\label{sect:DRside}

In this section, we show that the category of $G$-equivariant Lubin-Tate bundles on $\Omega$ defined by Kohlhaase is equivalent to the category of finite-dimensional smooth representations of $H$ over $K$.

\begin{defn}
    An object $\sV$ of $\Vect^{G}(\Omega)$ or $\VectCon^G(\Omega)$ is \emph{Lubin-Tate} if the restriction
    \[
    \sV |_{G_0} \in \Vect^{G_0}_{\LT}(\Omega) \subset \Vect^{G_0}(\Omega),
    \]
    and further that $\sV$ is \emph{Lubin-Tate of level $m$} if the same is true of $\sV|_{G_0}$. We let
\begin{align*}
    \VectCon^{G}_{\LT}(\Omega) &= \bigcup_{m \geq 1} \VectCon^{G}_{\LT,m}(\Omega), \\
    \Vect^{G}_{\LT}(\Omega) &= \bigcup_{m \geq 1} \Vect^{G}_{\LT,m}(\Omega),
\end{align*}
denote the corresponding full subcategories of $G$-equivariant Lubin-Tate bundles on $\Omega$.
\end{defn}

\begin{remark}
    The category $\Vect^G_{\LT}(\Omega)$ agrees with that defined by Kohlhaase \cite[Def.\ 4.2]{KOH}.
\end{remark}


To construct objects of $\Vect^G_{\LT}(\Omega)$ and $\VectCon^G_{\LT}(\Omega)$ one can use the functors of Section \ref{sect:generalities}.

\begin{defn}
    For $m \geq 1$, we set $L_m \coloneqq c(\sM_m)$.
\end{defn}

The space $\sM_m$ is the disjoint union of spaces
\[
\sM_m = \bigsqcup_{i \in \bZ} \sM_m^i,
\]
where $\sM_m^0 = \sN_m$, and the action of the uniformiser $\Pi$ of $D$ in $H$ induces $G^0$-equivariant isomorphisms $\Pi^i \colon \sM^i_m \xrightarrow{\sim} \sM^0_m = \sN_m$ \cite[\S 7.1]{TAY3}. From this,
\[
c(\sM_m^i) = \OO(\sM_m^i)^{G_m}
\]
because $c(\sN_m) = \OO(\sN_m)^{G_m}$ (see Section \ref{sect:DRside0}), and therefore
\[
L_m = c(\sM_m) = \prod_{i \in \bZ} c(\sM_m^i) = \prod_{i \in \bZ} \OO(\sM_m^i)^{G_m} = \OO(\sM_m)^{G_m}.
\]
In particular, the triple $(X,G,H) = (\sM_m, G \times H/H_m, G_m)$ fits into the formalism of Section \ref{sect:generalities} with $L_m$ equal to the $K$-algebra $L$ of Section \ref{sect:generalities} in either case (A) or (B).

For example, using this approach one obtains functors
\[
(\OO_{\sM_m} \otimes_{L_m} -)^H \colon \Rep_{L_m}^{H_m \times G_m}(G \times H) \rightarrow \Vect^{G}_{\LT, m}(\Omega).
\]

However, unlike with the $G^0$-equivariant Lubin-Tate bundles considered above, the subgroup $G_m \leq G$ is no longer normal, and so the functor (cf.\ Remark \ref{rmk:simplerform})
\[
(\OO(\sM_m) \otimes_{\OO(\Omega)} (-)(\Omega))^{G_m} \colon \Vect^{G}(\Omega) \rightarrow \Mod_{L_m}
\]
doesn't obviously factor through the forgetful map $\Rep_{L_m}^{H_m \times G_m}(G \times H) \rightarrow \Mod_{L_m}$ and give us a candidate inverse functor.

The key result of this section is an alternative definition of a $G$-equivariant Lubin-Tate bundle, which fixes this issue, and gives a semilinear representation theoretic description of this category.

\begin{defn}
    For $m \geq 1$, set $G^m = \det^{-1}(1 + \pi^m \OO_F)$.
\end{defn}

We have that $G^m \leq G^0$ and $G^m$ acts on $K_m$ through the determinant action considered at the start of Section \ref{sect:LTDRbundles}, and hence
\[
    \OO(\sN_m)^{G_m} = K_m = \OO(\sN_m)^{G^m}.
\]
The action of the uniformiser of $H$ induces $G^0$-equivariant isomorphisms $K_m \xrightarrow{\sim} c(\sM_m^i) $, and thus
\[
 \OO(\sM_m)^{G_m} = \prod_{i \in \bZ} \OO(\sM_m^i)^{G_m} = L_m = \prod_{i \in \bZ} \OO(\sM_m^i)^{G^m} =  \OO(\sM_m)^{G^m}.
\]
We can interpret this as a statement about Lubin-Tate bundles, namely that for any $m \geq 1$ the trivial Lubin-Tate bundle $\sV = \OO \in \Vect^G(\Omega)$ satisfies
\[
    (f_m^*\sV)(\sM_m)^{G_m} = \OO(\sM_m)^{G_m} = L_m = \OO(\sM_m)^{G^m} = (f_m^*\sV)(\sM_m)^{G^m}.
\]
The following asserts that this is actually true of \emph{any} Lubin-Tate bundle of level $m$.

\begin{lemma}\label{lem:equivdefs}
    Suppose that $\sV \in \Vect^{G}(\Omega)$, $m \geq 1$, and set 
    \[
        \sW \coloneqq f_m^* \sV \in \Vect^{G \times H / H_m}(\sM_m).
    \]    
    Then the following are equivalent: 
    \begin{enumerate}
        \item $\sV \in \Vect^{G}_{\LT,m}(\Omega)$,
        \item $\sV \in \Vect^{G}_{\LT,m}(\Omega)$ and $\sW(\sN_m)^{G_m} = \sW(\sN_m)^{G^m}$,
        \item The natural map
        \[
        \OO_{\sM_m} \otimes_{L_m} \sW(\sM_m)^{G^m} \rightarrow \sW
        \]
        is an isomorphism.
    \end{enumerate}
    The same holds with $\Vect$ replaced with $\VectCon$ throughout.
\end{lemma}

\begin{proof}
    Note that $\sV$ is Lubin-Tate if and only if the natural map
    \[
    \OO_{\sN_m} \otimes_{K_m} \sW(\sN_m)^{G_m} \rightarrow \sW|_{\sN_m}
    \]
    is an isomorphism, or equivalently if and only if the natural map
    \[
    \OO_{\sM_m} \otimes_{L_m} \sW(\sM_m)^{G_m} \rightarrow \sW
    \]
    is an isomorphism, as the action of a uniformiser of $H$ provides $G^0$-equivariant isomorphisms
    \[
        \sW(\sN_m) = \sW(\sM_m^0) \xrightarrow{\sim} \sW(\sM_m^i) 
    \]
    for all $i \in \bZ$, which are compatible with the actions of $\OO(\sN_m)$ and $\OO(\sM_m^i)$. In particular, it is immediate that $(3)$ implies $(1)$, as $\sW(\sN_m)^{G^m} \subset \sW(\sN_m)^{G_m}$, and that $(2)$ implies $(3)$.

    In order see that $(1)$ implies $(2)$, suppose that $\sV$ is Lubin-Tate of level $m$. We want to show that $\sW(\sN_m)^{G_m} = \sW(\sN_m)^{G^m}$. Note that because $G^m = \SL_n(F) G_m$, we are reduced to showing that if $z \in \sW(\sN_m)^{G_m}$, then $z$ is fixed by all $g \in \SL_n(F)$. For any $t \geq m$, we have a $K_t \rtimes G^0$-module 
    \[
    V_t \coloneqq (f_t^*\sV)(\sN_t),
    \]
    and, because these actions are compatible, a $K_{\infty} \rtimes G^0$-module
    \[
        V_{\infty} \coloneqq \varinjlim_{t \geq m} V_t.
    \]
    A priori, it is not clear that $\SL_n(F)$ even preserves $V_m^{G_m}$. However, the inclusion $V_m \hookrightarrow V_\infty$ is $G^0$-equivariant, and hence
    \[
    V_m^{G_m} \hookrightarrow V_{\infty}^{G_m} \hookrightarrow V_{\infty}^{G^0\text{-}\sm}.
    \]
    Therefore, if we can show that $\SL_n(F)$ acts trivially on $V_{\infty}^{G^0\text{-}\sm}$, it must also act trivially on $V_m^{G_m}$.

    From the definition $\sV$ being Lubin-Tate of level $m$, we have that
    \[
        \OO(\sN_m) \otimes_{K_m} V_m^{G_m} \xrightarrow{\sim} V_m
    \]
    and so applying $\OO(\sN_t) \otimes_{\OO(\sN_m)} - $,
    \[
         \OO(\sN_t) \otimes_{K_m} V_m^{G_m} \xrightarrow{\sim} \OO(\sN_t) \otimes_{\OO(\sN_m)} V_m.
    \]
    The right-hand side is nothing but $V_t$, and when we take the $G_t$-invariants this becomes
    \[
        K_t \otimes_{K_m} V_m^{G_m}  \xrightarrow{\sim} V_t^{G_t}
    \]  
    because $\OO(\sN_t)^{G_t} = K_t$ and the action of $G_t$ on $V_m^{G_m}$ is trivial. 
    In particular,
     \begin{alignat*}{2}
        V_{\infty}^{G^0\text{-}\sm} &= \makebox[10em][c]{$\bigcup_{r \geq m} V_{\infty}^{G_r}$} & &= \makebox[10em][c]{$\bigcup_{r \geq m} \left( \varinjlim_{t \geq m} V_t  \right)^{G_r}$} \\
        &= \makebox[10em][c]{$\varinjlim_{r,t \geq m} V_t^{G_r}$} & &= \makebox[10em][c]{$\varinjlim_{t \geq m} V_t^{G_t}$} \\
        &\cong \makebox[10em][c]{$\varinjlim_{t \geq m} K_t \otimes_{K_m} V_m^{G_m}$} & &\cong \makebox[10em][c]{$K_{\infty} \otimes_{K_m} V_m^{G_m}$},
    \end{alignat*}
    where we have used that the diagonal is cofinal in the direct limit. In particular, $V_{\infty}^{G^0\text{-}\sm}$ is a smooth $K_{\infty}[\SL_n(F)]$-module, which is free of finite rank over $K_{\infty}$, because the same is true of $V_m$ over $K_m$, and hence $\SL_n(F)$ has to act trivially by the following lemma.
\end{proof}

\begin{lemma}\label{lem:subgroupsSLnF}
    The only open normal subgroup of $\SL_n(F)$ is $\SL_n(F)$.
\end{lemma}

\begin{proof}
    Write $Z$ for the centre of $\SL_n(F)$. If $N$ is a normal subgroup of $\SL_n(F)$, then $NZ / Z$ is a normal subgroup of $\text{PSL}_n(F)$ which is simple \cite[Thm.\ 8.4, Thm.\ 9.3]{LANG}. If $NZ = Z$, then $N \subset Z$ and $N$ is finite, and thus $N$ cannot be open. If $NZ / Z = \text{PSL}_n(F)$, then $NZ = \SL_n(F)$ and $N$ is a finite index subgroup of $\SL_n(F)$ with abelian quotient $\SL_n(F) / N = NZ / N \cong Z / (Z \cap N)$. But this quotient must be trivial, as $\SL_n(F)$ is perfect \cite[Thm.\ 8.3, Thm.\ 9.2]{LANG}.
\end{proof}

\begin{defn}
For $m \geq 1$, we define $Y_m = G^m \times H_m$. 
\end{defn}

\begin{cor}\label{cor:maincorLT}
For any $m \geq 1$, the functors
\[
    (\OO_{\sM_m} \otimes_{L_m} -)^H \colon \Rep_{L_m}^{Y_m}(G \times H) \rightarrow \VectCon^{G}_{\LT,m}(\Omega) \rightarrow \Vect^{G}_{\LT,m}(\Omega)
\]
are equivalences of categories, and $\VectCon^{G}_{\LT,m}(\Omega)$ and $\Vect^{G}_{\LT,m}(\Omega)$ are closed under sub-quotients.

Furthermore, these are compatible with respect to the forgetful functor
\[
\Rep^{Y_m}_{L_m}(G \times H) \rightarrow \Rep^{Z_m}_{K_m}(G_0 \times H_0), \qquad V \mapsto e_{0} \cdot V
\]
where $e_0$ is the idempotent of $L_m$ corresponding to $K_m$, in the sense that the diagram
\[\begin{tikzcd}
	{\Rep^{Y_m}_{L_m}(G \times H)} & {\VectCon^{G}_{\LT,m}(\Omega)} & {\Vect^{G}_{\LT,m}(\Omega)} \\
	{\Rep^{Z_m}_{K_m}(G_0 \times H_0)} & {\VectCon^{G_0}_{\LT,m}(\Omega)} & {\Vect^{G_0}_{\LT,m}(\Omega)}
	\arrow[from=1-1, to=1-2]
	\arrow[from=1-1, to=2-1]
	\arrow[from=1-2, to=1-3]
	\arrow[from=1-2, to=2-2]
	\arrow[from=1-3, to=2-3]
	\arrow[from=2-1, to=2-2]
	\arrow[from=2-2, to=2-3]
\end{tikzcd}\]
commutes.
\end{cor}

\begin{proof}
For $\VectCon$, the functor is the composition of
\[
\OO_{\sM_m} \otimes_{L_m} - \colon \Rep^{Y_m}_{L_m}(G \times H) \rightarrow \VectCon^{G \times H / H_m}(\sM_m),
\]
from Section \ref{sect:generalities} (with $(X,G,H)$ taken as $(\sM_m, G \times H/H_m, G^m)$) and the equivalences
\[
(-)^H \colon \VectCon^{G \times H / H_m}(\sM_m) \xrightarrow{\sim} \VectCon^{G \times H / H_0}(\sM) \xrightarrow{\sim} \VectCon^{G}(\Omega)
\]
of \cite[Prop.\ 2.53]{TAY3} and \cite[\S 7.5]{TAY3}. The assumptions of Section \ref{sect:generalities} are satisfied, as $G \times H / H_m$ acts transitively on the components $(\sM_m^i)_{i \in \bZ}$ of $\sM_m$ \cite[\S 7.1]{TAY3}, and each connected component $X$ of $\sM_m$ is irreducible as a $G_m\text{-}\OO_{X}$-module by Corollary \ref{cor:irred} and the isomorphisms $\sM_m^i \xrightarrow{\sim} \sN_m$ induced by the action of a uniformiser of $H$. Therefore the first part of the statement follows from Proposition \ref{prop:mainprop} and Lemma \ref{lem:equivdefs}. The commutativity follows from the commutativity of
\[\begin{tikzcd}[ampersand replacement=\&]
	{\Rep^{Y_m}_{L_m}(G \times H)} \& {\VectCon^{G \times H / H_m}(\sM_m)} \\
	{\Rep^{Z_m}_{K_m}(G_0 \times H_0)} \& {\VectCon^{G_0 \times H_0 / H_m}(\sN_m)}
	\arrow[from=1-1, to=1-2]
	\arrow[from=1-1, to=2-1]
	\arrow[from=1-2, to=2-2]
	\arrow[from=2-1, to=2-2]
\end{tikzcd}\]
which commutes by construction, and the commutativity of
\[\begin{tikzcd}[ampersand replacement=\&]
	{\VectCon^{G \times H / H_m}(\sM_m)} \& {\VectCon^{G \times H / H_0}(\sM)} \& {\VectCon^{G}(\Omega)} \\
	{\VectCon^{G_0 \times H_0 / H_m}(\sN_m)} \& {\VectCon^{G_0}(\sN)} \& {\VectCon^{G_0}(\Omega)}
	\arrow["\sim", from=1-1, to=1-2]
	\arrow[from=1-1, to=2-1]
	\arrow["\sim", from=1-2, to=1-3]
	\arrow[from=1-2, to=2-2]
	\arrow[from=1-3, to=2-3]
	\arrow["\sim", from=2-1, to=2-2]
	\arrow["\sim", from=2-2, to=2-3]
\end{tikzcd}\]
which commutes because the left-hand square does by definition and the right-hand square does by \cite[Lem.\ 7.10]{TAY3}. The same holds with $\VectCon$ replaced by $\Vect$ everywhere.
\end{proof}

\begin{remark}
    For the equivalence of categories part of Corollary \ref{cor:maincorLT} one doesn't need the full force of the results of Section \ref{sect:generalities}, it being possible to show directly that $(\OO_{\sM_m} \otimes_{L_m} -)^H$ has inverse $(\OO(\sM_m) \otimes_{\OO(\Omega)} (-)(\Omega))^{G^m}$, the key property used being the characterisation of Lemma \ref{lem:equivdefs}.
\end{remark}

For $m' \geq m$, there is a fully faithful inclusion functor
\[
L_{m'} \otimes_{L_m} - \colon \Rep^{Y_m}_{L_m}(G \times H) \rightarrow \Rep^{Y_{m'}}_{L_{m'}}(G \times H),
\]
and this is compatible with inclusion of Lubin-Tate bundles of levels $m$ and $m'$ in the sense that
\[\begin{tikzcd}
	{\Rep^{Z_{m'}}_{L_{m'}}(G \times H)} & {\VectCon^{G}_{\LT,m'}(\Omega)} & {\Vect^{G}_{\LT,m'}(\Omega)} \\
	{\Rep^{Z_m}_{L_m}(G \times H) } & {\VectCon^{G}_{\LT,m}(\Omega)} & {\Vect^{G}_{\LT,m}(\Omega)}
	\arrow["\sim", from=1-1, to=1-2]
	\arrow["\sim", from=1-2, to=1-3]
	\arrow[hook, from=2-1, to=1-1]
	\arrow["\sim", from=2-1, to=2-2]
	\arrow[hook, from=2-2, to=1-2]
	\arrow["\sim", from=2-2, to=2-3]
	\arrow[hook, from=2-3, to=1-3]
\end{tikzcd}\]
commutes.

\begin{defn}
    We define 
    \begin{align*}
    \Psi_{G} &\colon \Rep_K^{\sm}(G) \rightarrow \VectCon_{\LT}^{G}(\Omega),\\
    \Psi_{H} &\colon \Rep_K^{\sm}(H) \rightarrow \VectCon_{\LT}^{G}(\Omega),
    \end{align*}
    as the direct limits of the functors
    \begin{align*}
    \Rep_K^{G^m}(G) \xrightarrow{L_m \otimes_K -} \Rep_{L_m}^{Y_m}(G \times H) \rightarrow \VectCon^{G}_{\LT,m}(\Omega),\\
    \Rep_K^{H^m}(H) \xrightarrow{L_m \otimes_K -} \Rep_{L_m}^{Y_m}(G \times H) \rightarrow \VectCon^{G}_{\LT,m}(\Omega).
    \end{align*}
\end{defn}

Here we have used that the categories $\Rep_K^{G^m}(G)$ exhaust $\Rep_K^{\sm}(G)$, which is true because $\Rep_K^{G_m}(G) = \Rep_K^{G^m}(G)$ as a consequence of the following lemma.

\begin{lemma}\label{lem:normaliserGm}
    The normal closure of $G_m$ in $G$ is $G^m$.
\end{lemma}

\begin{proof}
    Let $N$ be the normal closure of $G_m$ in $G$. Then $N \cap \SL_n(F)$ is an open normal subgroup of $\SL_n(F)$, hence $\SL_n(F) \subset N$ by Lemma \ref{lem:subgroupsSLnF}, and thus $G^m \subset N$, as $G^m = \SL_n(F)G_m$ because the determinant map $\det \colon G_m \rightarrow 1 + \pi^m \OO_F$ is surjective.
\end{proof}

\begin{remark}\label{rem:explicitformGH}
By abuse of notation we will also view $\Psi_{G}$ and $\Psi_H$ as functors to $\Vect^{G}_{\LT}(\Omega)$.
\end{remark}

Kohlhaase also considered functors from $\Rep_K^{\sm}(G)$ and $\Rep_K^{\sm}(H)$ to $\Vect_{\LT}^{G}(\Omega)$ \cite[Thm.\ 4.5]{KOH}, using the equivalences,
\[
\bD_{\LT} \colon \Vect^{G}_{\LT}(\Omega) \xleftrightarrow{\sim} \Vect^{H}_{\Dr}(\bP^{n-1}) : \! \bD_{\Dr},
\]
where $\Vect^{H}_{\Dr}(\bP^{n-1})$ is an analogously defined category of $H$-equivariant \emph{Drinfeld bundles} on $\bP^{n-1}$. The following shows that these coincide with $\Psi_{G}$ and $\Psi_{H}$.

\begin{lemma}\label{lem:explicitformofGH}
    For $V \in \Rep_K^{\sm}(G)$ and $W \in \Rep_K^{\sm}(H)$ there are natural isomorphisms,
    \begin{align*}
        \Psi_{G}(V) &\cong \OO_{\Omega} \otimes_K V, \\
        \Psi_{H}(W) &\cong \bD_{\Dr}(\OO_{\bP^{n-1}} \otimes_K W)
    \end{align*}
    in $\Vect_{\LT}^{G}(\Omega)$.
\end{lemma}

\begin{proof}
    If $G^m$ acts trivially on $V$, then
    \begin{align*}
    \Psi_{G}(V)
    &= (\OO_{\sM_m} \otimes_K V)^{H/H_m}, \\
    &= \OO_{\sM_m}^{H/H_m} \otimes_K V, \\
    &= \OO_{\Omega} \otimes_K V
    \end{align*}
    because $H$ acts trivially on $V$. For the description of $\Psi_{H}(W)$, as all spaces are the disjoint union of quasi-Stein spaces it is sufficient to give the isomorphism on global sections. As $W$ is smooth some $H_m$ acts trivially on $W$, and from \cite[\S 4 (26)]{KOH} we have that (in the notation of \cite{KOH}),
    \begin{align*}
        \Gamma(\Omega, \bD_{\Dr}(\OO_{\bP^{n-1}} \otimes_K W))
        &= (\mathbb{C}_{\infty} \otimes_{K} W)^{H}, \\
        &\equiv (\mathbb{C}_{\infty}^{H_m} \otimes_{K} W)^{H/H_m}, \\
        &\equiv (\mathbb{A}_{m} \otimes_{K} W)^{H/H_m},\\
        &= \Gamma(\Omega, \Psi_G(V))
    \end{align*}
    using \cite[Lem.\ 3.3]{KOH} for the second equality, and that $\mathbb{C}_{\infty}^{H_m} = \mathbb{A}_m$ \cite[Thm.\ 4.3]{KOH} for the third.
\end{proof}

\begin{prop}
    The functors $\Psi_{G}$ and $\Psi_{H}$ are fully faithful and preserve irreducibility.
\end{prop}
\begin{proof}
    Each is the composition of fully faithful functors. The functors
    \begin{align*}
    L_m \otimes_K - \colon \Rep_K^{G^m}(G) \rightarrow \Rep_{L_m}^{Y_m}(G \times H),\\
    L_m \otimes_K - \colon \Rep_K^{H^m}(H) \rightarrow \Rep_{L_m}^{Y_m}(G \times H),
    \end{align*}
    both preserve irreducibility, by Corollary \ref{cor:repsfunctors} with $(F,G,H)$ taken as $(L_m, G / G^m \times H / H_m, H/H_m)$ and $(L_m, G / G^m \times H / H_m, G/G^m)$ respectively, and the functors $\Rep_{L_m}^{Y_m}(G \times H) \rightarrow \VectCon^G_{\LT,m}(\Omega)$ and $\Rep_{L_m}^{Y_m}(G \times H) \rightarrow \Vect^G_{\LT,m}(\Omega)$ preserve irreducibility by Corollary \ref{cor:maincorLT}.
\end{proof}

It turns out that $\Psi_H$ is actually an equivalence of categories.

\begin{thm}\label{thm:finalmainthmLT}
    The functors
    \[
    \Rep_K^{\sm}(H) \xrightarrow{\Psi_{H}} \VectCon^{G}_{\LT}(\Omega) \rightarrow \Vect^{G}_{\LT}(\Omega)
    \]
    are equivalences of categories.
\end{thm}

\begin{proof}
    This follows from Corollary \ref{cor:maincorLT}, once we show that each
    \[
    L_{m} \otimes_K - \colon \Rep_K^{H_m}(H) \rightarrow \Rep^{Y_m}_{L_{m}}(G \times H) 
    \]
    is an equivalence. To see this, set
    \begin{align*}
        X_0 = \bigsqcup_{i \in \bZ} \Spec(c(\sM^i)), \qquad X_m = \bigsqcup_{i \in \bZ} \Spec(c(\sM^i_m)),
    \end{align*}
    and note that there are are pullback functors
    \[
    \Vect^{H/H_m}(\Spec(K)) \rightarrow \Vect^{G / G^0 \times H/H_m}(X_0) \rightarrow \Vect^{G / G^{m} \times H/H_m}(X_m),
    \]
    which are equivalences: the second as $G^0 / G^m$ is the Galois group of $X_m / X_0$ (\cite[Prop.\ 2.53]{TAY3}) and the first by \cite[Ex.\ 2.33]{TAY3}, using the fact that $G / G^0$ acts on the connected components of $X_0$ simply transitively. Taking global sections, this becomes
    \[
    \Rep_K^{H_m}(H) \xrightarrow{\sim} \Rep_{c(\sM)}^{G^0 \times H_m}(G \times H) \xrightarrow{\sim} \Rep_{L_m}^{Y_m}(G \times H)
    \]
    the composition of which coincides with $L_m \otimes_K -$ above. 
\end{proof}

\begin{remark}\label{rem:PsiGnotequivingeneral}
    Similarly, one can show $\Psi_G$ is also an equivalence when $n = 1$ (cf.\ \cite[Prop.\ 4.6]{KOH}).
    
    When $n \geq 2$, however, $\Psi_G$ is not an equivalence. To construct an object outside the essential image, let $W \in \Rep_K^{\sm}(H)$ be any representation of $H$ which is non-trivial on $\SL_1(D)$, the kernel of the reduced norm $\Nrd \colon H \rightarrow F^{\times}$. For example one could take $W$ to be any of the representations $V_m$ of \cite[Ex.\ 7.17]{TAY3}, where the condition that $n \geq 2$ is used to ensure that $\SL_1(D) \neq 0$.
    
    Let $L$ be an extension of $K$ containing $F^{\ab}$. Because $\Psi_H$ commutes with base change (by Remark \ref{rem:dualfunctorcompat} and \cite[Lem.\ B.2]{TAY3}), $\Psi_H(W)_L \cong \Psi_{H, L}(W_L)$, and by \cite[Thm.\ C]{TAY3} $\Psi_{H, L}(W_L)$ is non-trivial when restricted to $\VectCon(\Omega)$, in the sense that it is not the direct sum of $\dim_K(W)$ copies of $\OO_{\Omega}$.

    On the other hand, by Lemma \ref{lem:explicitformofGH} we have that for every $V \in \Rep_K^{\sm}(G)$, $\Psi_G(V)_L = \OO_{\Omega} \otimes_K V = \OO_{\Omega, L} \otimes_L V_L$, and therefore $\Psi_G(V)_L$ is isomorphic to the direct sum of $\dim_K(V)$ copies of $\OO_{\Omega_L}$ when restricted to $\VectCon(\Omega)$.

    In particular, $\Psi_{H}(W) \not\cong \Psi_G(V)$ for any $V \in \Rep_K^{\sm}(G)$, and therefore $\Psi_{H}(W)$ is not in the essential image of $\Psi_G$.

    

    
\end{remark}

\begin{cor}
    The functor
    \[
    \OO_{\bP^{n-1}} \otimes_K - \colon \Rep_K^{\sm}(H) \rightarrow \Vect^H_{\Dr}(\bP^{n-1})
    \]
    is an equivalence of categories.
\end{cor}

\begin{proof}
    This follows directly from Lemma \ref{lem:explicitformofGH} and Theorem \ref{thm:finalmainthmLT}.
\end{proof}

\section{$G^0$-Finite Vector Bundles on $\Omega$}\label{sect:G0finite}

In this section we consider the functors
\begin{align*}
    \Phi_H \coloneqq \Hom_{H}(-, f_* \OO_{\sM_{\infty}}) &\colon \Rep_K^{\sm}(H) \rightarrow \Vect^G(\Omega), \\
    \Phi_H^0 \coloneqq \Hom_{H_0}(-, f_* \OO_{\sN_{\infty}}) &\colon \Rep_K^{\sm}(H_0) \rightarrow \Vect^{G^0}(\Omega),
\end{align*}
which are defined in the same way as \cite[\S7.7, \S7.8]{TAY3}. We describe how one can use the results of Section \ref{sect:generalities} and Section \ref{sect:infintesimalaction} to show that main result of \cite{TAY3} also holds for these functors.

\begin{remark}\label{rem:dualfunctorcompat}
    For $V \in \Rep_K^{\sm}(H)$ and $W \in \Rep_K^{\sm}(H_0)$, there are natural isomorphisms
    \begin{align*}
    \Phi_H(V) &\cong \Psi_H(V^*), \\
    \Phi_H^0(W) &\cong \Psi^0_H(W^*)
    \end{align*}
    as $G$-equivariant (resp. $G_0$-equivariant) vector bundles (with connection). This follows from Remark \ref{rem:explicitformGH} together with \cite[Thm.\ 6.1(1)]{TAY3} and the construction of $\Phi_H^0$ and $\Phi_H$.
\end{remark}

\begin{thm}\label{thm:finiteVB1}
    The functors $\Phi_H^0$ and $\Phi_H$ are exact, monoidal, and fully faithful. The diagram
\[\begin{tikzcd}
	{\Rep_K^{\sm}(H)} & {\Vect^{G}(\Omega)} \\
	{\Rep_K^{\sm}(H_0)} & {\Vect^{G^0}(\Omega)}
	\arrow["{\Phi_H}", from=1-1, to=1-2]
	\arrow[from=1-1, to=2-1]
	\arrow[from=1-2, to=2-2]
	\arrow["{\Phi_{H}^0}", from=2-1, to=2-2]
\end{tikzcd}\]
    commutes, and the essential image of each is closed under sub-quotients.
\end{thm}

\begin{proof}
    The triples $(\sN_m, G^0 \times H_0 / H_m, G^0)$ and $(\sM_m, G \times H / H_m, G)$ satisfy the assumptions of Section \ref{sect:generalities}, which can be checked similarly to the start of Section \ref{sect:DRside0} and the start of Section \ref{sect:DRside}. The fact that $\Phi_H^0$ and $\Phi_H$ are exact, monoidal, fully faithful and have essential images closed under sub-quotients then follows exactly as in \cite[\S7.10]{TAY3}, which only uses \cite[\S 6]{TAY3}, which itself goes through with $\VectCon$ replaced with $\Vect$ by appealing to the results Section \ref{sect:generalities} in lieu of the results of \cite[\S 4]{TAY3}. The compatibility between $\Phi^0_H$ with $\Phi_H$ follows analogously to the case for $\VectCon$ \cite[\S 7.8]{TAY3}.
\end{proof}

We can also describe the essential images of $\Phi_H^0$ and $\Phi_H$, by modifying the methods of \cite{TAY3} to our setting. The main difference is that unlike $\VectCon^{G^0}(\Omega)$, the category $\Vect^{G^0}(\Omega)$ is not obviously abelian (it might fail to be closed under quotients), and therefore one needs to also use the larger abelian category $\Coh^{G^0}(\Omega)$.

We need a couple of notions from \cite[\S 5]{TAY3}. For a set of objects $\sS$ of $\Coh^{G^0}(\Omega)$, we follow \cite[\S 3]{TAY3} and write $\sC(\sS)$ for the full subcategory of $\Coh^{G^0}(\Omega)$ whose objects are those objects isomorphic to a quotient of a sub-object of a direct sum of objects of $\sS$. For $\sV \in \Coh^{G^0}(\Omega)$, we set $\sC(\sV) \coloneqq \sC(\sS_{\sV})$ where $\sS_{\sV}$ is the set of objects $\{(\sV)^{\otimes s} \otimes (\sV^*)^{\otimes t}\}_{s,t \in \bZ_{\geq 0}}$. For $f  = \sum_{n \geq 0} a_n x^n \in \bZ_{\geq 0}[x]$ and $\sV \in \Vect^{G^0}(\Omega)$, we define
\[
	f(\sV) \coloneqq \bigoplus_{n \geq 0} (\sV^{\otimes n})^{\oplus a_n} \in \Vect^{G^0}(\Omega).
\]
\begin{defn}\label{defn:finiteobj}
An object $\sV \in \Vect^{G^0}(\Omega)$ is \emph{finite} if
\begin{enumerate}
    \item there are $f,g \in \bZ_{\geq 0}[x]$ with $f \neq g$ and $f(\sV) \cong g(\sV)$,
    \item for any complete field extension $L$ of $K$, $\sC(\sV_L) \subset \Vect^{G^0}(\Omega_L)$.
\end{enumerate}

We write $\Vect^{G^0}(\Omega)_{\fin}$ for the full subcategory of $\Vect^{G^0}(\Omega)$ whose objects are the finite objects, and $\Vect^{G}(\Omega)_{G^0\text{-}\fin}$ for the full subcategory of $\Vect^{G}(\Omega)$ whose objects are those that restrict to a finite object of $\Vect^{G^0}(\Omega)$.
\end{defn}

\begin{thm}\label{thm:finiteVB2}
    The functors $\Phi_H^0$ and $\Phi_H$ induce equivalences
\[\begin{tikzcd}
	{\Rep_K^{\sm}(H)} & {\Vect^{G}(\Omega)_{G^0\text{-}\fin}} \\
	{\Rep_K^{\sm}(H_0)} & {\Vect^{G^0}(\Omega)_{\fin}}
	\arrow["{\Phi_H}", from=1-1, to=1-2]
	\arrow[from=1-1, to=2-1]
	\arrow["\sim", draw=none, from=1-2, to=1-1]
	\arrow[from=1-2, to=2-2]
	\arrow["{\Phi_{H}^0}", from=2-1, to=2-2]
	\arrow["\sim", draw=none, from=2-2, to=2-1]
\end{tikzcd}\]
\end{thm}

\begin{proof}
   Given $V \in \Rep_K^{\sm}(H_0)$, $\Phi^0_H(V)$ is finite: $\sV$ satisfies the required polynomial relation by \cite[Thm.\ 6.1(4)]{TAY3}, and for any complete extension $L / K$, $\Phi^0_H$ commutes with base change by \cite[Lem.\ B.2]{TAY3} and so
   \[
   \sC(\Phi^0_H(V)_L) = \sC(\Phi^0_{H, L}(V_L)) \subset \Phi^0_{H,L}(\Rep_L^{\sm}(H)) \subset \Vect^{G^0}(\Omega_L),
   \]
   the essential image of $\Phi^0_{H,L}$ being closed under direct sums, duals, tensor products and sub-quotients.

    Conversely given $\sV \in \Vect^{G^0}(\Omega)$ which is finite, we may reason following \cite[\S 7.11]{TAY3}. For any complete field extension $L / K$ the full subcategory $\sC(\sV_L)$ is abelian by assumption, and being contained in $\Vect^{G^0}(\Omega_L)$ can be viewed as a Tannakian category by fixing an $L$-rational point $z$ of $\Omega_L$ and using the fibre functor
\[
\omega_z \colon \sC(\sV_L) \rightarrow \Vect_L, \qquad \omega_z \colon \sV \mapsto \sV(z) \coloneqq \sV_z \otimes_{\OO_{\Omega_L,z}} k(z).
\]
    Here we use the condition that $\sC(\sV_L) \subset \Vect^{G^0}(\Omega_L)$ to ensure that $\omega_z$ is exact. In particular, when $L = C$ is a complete algebraically closed extension of $K$ we can follow exactly the same reasoning as \cite[Prop.\ 5.8]{TAY3} to see that $\sC(\sV_C)$ is equivalent to the category of $C$-representations of a constant group scheme, and use this to build a $G^0$-equivariant finite \'{e}tale Galois covering of $f \colon Z \rightarrow \Omega_C$ with $\sV_C$ a direct summand of $f_* \OO_Z$. One can then reason exactly as in the proof of \cite[Thm.\ 7.18]{TAY3}: using the Scholze-Weinstein factorisation theorem \cite[Thm.\ 7.3.1]{SW} one can realise $\sV_C$ as an object in the essential image of $\Phi^0_{H,C}$, and from this deduce that $\sV$ is in the essential image of $\Phi^0_H$ from the compatibility of $\Phi^0_H$ with base change.

    The description of the essential image of $\Phi_H$ follows from that of $\Phi_H^0$ as in \cite[Cor.\ 7.19]{TAY3}.
\end{proof}

From Theorem \ref{thm:finiteVB2} and Theorems A and B of \cite{TAY3}, we immediately have the following.

\begin{cor}\label{cor:forgetfulequiv}
The forgetful functors
\begin{align*}
    \VectCon^{G}(\Omega)_{G^0\text{-}\fin} &\xrightarrow{\sim} \Vect^{G}(\Omega)_{G^0\text{-}\fin}, \\
    \VectCon^{G^0}(\Omega)_{\fin} &\xrightarrow{\sim} \Vect^{G^0}(\Omega)_{\fin}
\end{align*}
are equivalences of categories.
\end{cor}

The rank $1$ finite objects of $\VectCon^{G^0}(\Omega)$ are the torsion line bundles \cite[Lem.\ 5.5]{TAY3}. We can show this is also true in $\Vect^{G^0}(\Omega)$ with our slightly modified definition of finiteness.

\begin{prop}\label{prop:torsionLB}
Suppose that $\sL \in \Vect^{G^0}(\Omega)$ has rank $1$. Then $\sL$ is finite if and only if $\sL^{\otimes m} \cong \OO_{\Omega}$ for some $m \geq 1$.
\end{prop}

\begin{proof}
    If $\sL$ is finite, then applying \cite[Lem.\ 5.5]{TAY3} to the Tannakian category $\sC(\sL)$ (after choosing an $K$-point of $\Omega$) we have that $\sL^{\otimes m} \cong \OO_{\Omega}$ for some $m \geq 1$.
    
    Conversely, suppose that $\sL^{\otimes m} \cong \OO_{\Omega}$ for some $m \geq 1$, so $f(\sL) = g(\sL)$ for $f(x) = x^m$ and $g(x) = 1$. To verify the second condition, let $L$ be a complete extension of $K$, and note that as in the first part of the proof of Theorem \ref{thm:finiteVB2}, to show that every object of $\sC(\sL_L)$ is a vector bundle it is enough to show that $\sL_L$ is in the essential image of $\Phi^0_{H,L}$.
    
    Set $L_m \coloneqq L(\mu_m)$ be the finite extension of $L$ defined by adjoining the $m$th roots of $1$ to $L$, and fix an isomorphism $\alpha \colon \sL_{L_m}^{\otimes m} \xrightarrow{\sim} \OO_{\Omega_L}$. To this pair $(\sL_{L_m}, \alpha)$, we may follow the construction of \cite[pg. 11]{TAY2} (appropriately adding $G^0$-equivariant structure) to produce a $G^0$-equivariant finite \'{e}tale Galois covering $f \colon Z \rightarrow \Omega_{L_m}$ with Galois group $\mu_m(L_m)$ and with $\sL_{L_m}$ a direct summand of $f_* \OO_Z$.
    
    The proof of \cite[Thm.\ 7.18]{TAY3} starts with exactly such a covering, and thus by the same proof $\sL_{L_m}$ is in the essential image of $\Phi^0_{H,L_m}$. Then $\sL_L$ itself is in the essential image of $\Phi^0_{H,L}$, using the compatibility of $\Phi^0_{H,L}$ with base change as in the final part of the proof of \cite[Thm.\ 7.18]{TAY3}.
\end{proof}

\section{Properties of $\Phi_H$}\label{sect:properties}

For any irreducible $V \in \Rep_{\sm}(H)$ with $\dim(V) > 1$, it is expected that
\[
\HH_{\dRc}^{n-1}(\Omega, \Phi_H(V)) \cong \JL(V)^{\oplus n}
\]
in $\Rep^{\sm}_K(G)$ (when $K$ is algebraically closed), where
\[
\JL \colon \Irr^{\sm}_K(H) \xrightarrow{\sim} \Irr^{\sm, \ESI}_K(G)
\]
is the Jacquet-Langlands correspondence from smooth irreducible representations of $H$ to irreducible smooth essentially square-integrable representations of $G$. For example, this is known in dimension $1$ for $\GL_2(F)$ \cite[Thm.\ 0.4]{CDN1}, and in any dimension for certain level $0$ representations of $H$ corresponding to the first Drinfeld covering $\sM_1$ \cite[Thm.\ A]{JUNDR}.

It is therefore natural to ask which properties of the Jacquet-Langlands correspondence can already be seen from the properties of the functor $\Phi_H$. For example, the Jacquet-Langlands correspondence satisfies three natural properties:
\begin{enumerate}
    \item $\JL(V^*) = \JL(V)^*$ ,
    \item $\omega_{V} = \omega_{\JL(V)}$,
    \item $\JL(\chi_H \otimes V) = \chi_G \otimes \JL(\rho)$,
\end{enumerate}
where $\omega_V$ denotes the central character of $V$, $\chi \colon F^\times \rightarrow K^\times$ is a smooth character of $F^\times$, and 
\begin{align*}
    (-)_G &\colon \Char^{\sm}_K(F^\times) \xrightarrow{\sim} \Char^{\sm}_K(G), \qquad \chi_G = \chi \circ \det, \\
    (-)_H &\colon \Char^{\sm}_K(F^\times) \xrightarrow{\sim} \Char^{\sm}_K(H), \qquad \chi_H = \chi \circ \Nrd,
\end{align*}
are the associated smooth characters of $G$ and $H$ respectively. These are well-defined bijections because both $\det \colon G \rightarrow F^{\times}$ and $\Nrd \colon H \rightarrow F^{\times}$ are continuous, surjective, and are abelianizations of $G$ and $H$ respectively (by \cite[Thm.\ 8.3, Thm.\ 9.2]{LANG} and \cite[\S 1.4.3]{PRR}).

This motivates the following Theorem, which establishes that the corresponding properties hold true of the functor $\Phi_H$ (which we consider as a functor to either $\VectCon^G(\Omega)$ or $\Vect^G(\Omega)$).

\begin{thm}\label{thm:twists}
For any $V \in \Rep_K^{\sm}(H)$,
\begin{enumerate}
    \item $\Phi_H(V^*) = \Phi_H(V)^*$ ,
    \item If $V$ admits a central character $\omega_V$, then $\omega_{V} = \omega_{\Phi_H(V)}^*$,
    \item For any smooth character $\chi \colon F^\times \rightarrow K^\times$,
    \[
    \Phi_H(V \otimes \chi_H) = \Phi_H(V) \otimes \chi_G^*.
    \]
\end{enumerate}
\end{thm}
\begin{remark}
Here, for $\sV \in \VectCon^{\GL_n(F)}(\Omega)$ with endomorphism ring $K$, $\omega_{\sV}$ is the character
\[
\omega_{\sV} \coloneqq (-)^{\sV} \colon F^\times \equiv Z(\GL_n(F)) \rightarrow \Aut(\sV) = K^\times,
\]
where for $g \in \GL_n(F)$, $g^{\sV} \colon \sV \rightarrow g^{-1} \sV$ is the equivariant structure morphism, which is an automorphism of $\sV$ whenever $g$ is central, because $g$ is central and acts trivially on $\Omega$.
\end{remark}

\begin{remark}
    We do not prove it here, but the analogue of Part (3) of Theorem \ref{thm:twists} for $\Psi_H^0$ is also true: for any smooth character $\lambda \colon \OO_F^{\times} \rightarrow K^\times$, $\Psi_H^0(V \otimes \lambda_H) = \Psi_H^0(V) \otimes \lambda_G$ (cf.\ Remark \ref{rem:commchar0}).
\end{remark}

\begin{remark}
     In particular, as
\[
H^{n-1}_{\dRc}(\Omega, \OO_{\Omega}) = \text{St}_G = \JL(1)
\]
\cite[\S 3 Thm.\ 1, \S 4 Lem.\ 1]{SchStu}, this shows that for any smooth character $\chi \colon F^\times \rightarrow K^{\times}$,
\[
H^{n-1}_{\dRc}(\Omega, \Phi_H(\chi_H)) = \text{St}_G \otimes \chi_G = \JL(\chi_H),
\]
which complements the expectation outlined above in the case that $\dim(V) = 1$.
\end{remark}

\begin{proof}
    The first point follows from the fact that $\Phi_H$ commutes with duals \cite[Rem.\ 7.14]{TAY3}. For the second point, suppose $\lambda \in F^\times$, write $\iota \colon F^\times \hookrightarrow H$, $d \colon F^\times \hookrightarrow G$ for the canonical inclusions, and let $m \geq 1$ be the level of $V$. We want to show that the image of the action of $\iota(\lambda) \in H$ on $V$ under $\Aut(V) \xrightarrow{\sim} \Aut(\Phi_H(V))$ is the inverse of the action of $d(\lambda) \in G$ on $\Phi_H(V)$, so is equal to $d(\lambda^{-1})^{\Phi_H(V)}$. For a local section $[\phi \colon V \rightarrow \OO_{\sM_m}(f^{-1}(U))] \in \Phi_H(V)(U)$, the functorially induced action of $\iota(\lambda)_V$ on $\phi$ is $\phi \circ \iota(\lambda)_V$ which equals $\iota(\lambda)^{\OO_{\sM_m}}_{f^{-1}(U)} \circ \phi$ as $\phi$ is $H$-equivariant. On the other hand, 
       \begin{align*}
       d(\lambda^{-1})^{\Phi_H(V)}_{U}(\phi) &= d(\lambda^{-1})^{\OO_{\sM_m}}_{f^{-1}(U)} \circ \phi, \\
       &= \iota(\lambda)^{\OO_{\sM_m}}_{f^{-1}(U)} \circ \phi,
       \end{align*}
    the first equality by definition and the second because the diagonal subgroup $F \hookrightarrow G \times H$, $\lambda \mapsto (\iota(\lambda), d(\lambda))$ acts trivially on $\sM_m$. For the third point, as $\Phi_H$ is monoidal it is sufficient to show that
    \[
    \Phi_H(\chi_H) = \Phi_H(1) \otimes \chi_G^* =  \OO_{\Omega} \otimes_K \chi_G^*.
    \]
    Writing $m \geq 1$ for the level of $\chi$, set $L$ to be any finite extension of $K$ which contains $\breve{F}_m$. Over $L$ the covering $\sM_{m,L}$ breaks up into geometrically connected components, and therefore the diagram of \cite[Thm.\ C]{TAY3} commutes on representations of level $m$ (which can be seen from the proof, which takes place at each level). In particular, as the restriction of $\chi_H \otimes_K L$ to $\SL_1(D)$ is trivial, $\Phi_H(\chi_H \otimes_K L) \cong \OO_{\Omega, L}$ as $\sD$-modules on $\Omega_L$. We have that 
    \[
    \Phi_H(\chi_H \otimes_K L) = \Phi_H(\chi_H)_L,
    \]
    as $\Phi_H$ commutes with base change \cite[Lem.\ B.2]{TAY3}, and thus by Lemma \ref{lem:basechangeinj} below, $\Phi_H(\chi_H) \cong \OO_{\Omega}$ as $\sD$-modules over $K$. Consequently, $\Phi_H(\chi_H) = \OO_{\Omega} \otimes_K \psi_G$ for some character $\psi$ of $F^\times$ by \cite[Prop. 3.2.14]{AW2} and the fact that $\SL_n(F)$ is perfect \cite[Thm.\ 8.3, Thm.\ 9.2]{LANG}, and we are reduced to showing that $\psi = \chi^*$.
    
    Let $C$ be a complete algebraically closed extension of $K$. If $g \in G_{\reg}^{\ellip}$ and $h \in H_{\reg}^{\ellip}$ are regular elliptic elements with the same characteristic polynomial, then for any $x \in \Omega(C)$ which is fixed by $g$ (which exists as $g$ is regular elliptic), there is an equality
    \[
    \tr_{\Psi_H(\chi_H) \otimes \kappa(x)}(g) = \tr_{\chi_H}(h) = \chi(\Nrd(h))
    \]
    by \cite[Thm.\ 4.7]{KOH} and Lemma \ref{lem:explicitformofGH}, noting that although \cite{KOH} is written in the case that $K = \breve{F}$, all constructions generalise without issue to the relative setting. We have that $\Psi(\chi_H) = \Phi_H(\chi_H^*)$ by Remark \ref{rem:dualfunctorcompat}, and the left-hand side is the trace of the action of $g$ on the $k(x)$-vector space
    \[
    \Phi_H(\chi_H^*) \otimes \kappa(x) = (\OO_{\Omega} \otimes_K \psi_G^* ) \otimes \kappa(x) = \kappa(x) \otimes_K \psi_G^*,
    \]
    or in other words $\psi^*(\det(g))$. As $g$ and $h$ have the same characteristic polynomial, $\det(g) = \Nrd(h)$, so $\chi(\Nrd(h)) = \psi^*(\Nrd(h))$, from which the fact that $\chi = \psi^*$ follows from the Lemma \ref{lem:surjregell} below.
\end{proof}

\begin{lemma}\label{lem:basechangeinj}
    For any finite extension $L$ of $K$, the base change homomorphism
    \[
     (-)_L \colon \PicCon(\Omega) \rightarrow \PicCon(\Omega_L)
    \]
    is injective.
\end{lemma}

\begin{proof}
    For any line bundle with connection on $\Omega$ and any finite extension $L$ of $K$, there is an isomorphism
    \[
    \sL(\Omega)^{\sT(\Omega) = 0} \otimes_K L \xrightarrow{\sim} \sL_L(\Omega)^{\sT(\Omega_L) = 0}.
    \]
    This follows from the same argument given in the proof of \cite[Lem.\ 3.5]{TAY3} with the sheaf $c_X$ replaced by $\Hom_{\sD}(\OO, \sL)$, the sections of which are described by \cite[Lem.\ 3.1.4]{AW2}. Using this, we see that the base change homomorphism
    \[
     (-)_L \colon \PicCon(\Omega) \rightarrow \PicCon(\Omega_L)
    \]
    is injective by \cite[Cor.\ 3.1.7]{AW2}, which shows that a line bundle with connection is trivial if and only if it possesses a non-zero global horizontal section. 
\end{proof}

\begin{lemma}\label{lem:surjregell}
    The reduced normal map induces a surjection
    \[
    \Nrd \colon H_{\reg}^{\ellip} \rightarrow F^\times.
    \]
\end{lemma}

\begin{proof}
    For $h \in H$, the reduced characteristic polynomial $f_h$ of $h$ is a power of the minimal polynomial $h$ over $F$ \cite[Lem.\ IV.2.4]{CSA}. In particular, if $f_h$ has distinct roots over $\overline{F}$ then $f_h$ is irreducible, and so any regular element of $H$ is regular elliptic \cite[\S 1.1]{BHJL}: $H_{\reg}^{\ellip} = H_{\reg}$. The set of $H_{\reg}$ is dense in $H$ \cite[Lem.\ A.3]{BHLTL}, and therefore the same is true for $H_{\reg}^{\ellip}$. The $n$th power subgroup $F^{\times n}$ of $F^\times$ is open, and hence
    \[
    H_{\reg}^{\ellip} \xrightarrow{\Nrd} F^\times \rightarrow F^{\times} / F^{\times n}
    \]
    is continuous to a discrete space and thus surjective, as $H_{\reg}^{\ellip}$ is dense in $H$ and $\Nrd \colon H \rightarrow F^\times$ is surjective. Therefore, as the set $H_{\reg}^{\ellip}$ is closed under multiplication by $F^\times$ (which follows directly from the definition) and $\Nrd(a) = a^n$ for $a \in F^\times$, $\Nrd \colon H_{\reg}^{\ellip} \rightarrow F^\times$ is also surjective. 
\end{proof}

\begin{remark}
    In light of \cite[\S 1.3]{BHJL}, Lemma \ref{lem:surjregell} is equivalent to the claim that for any $n \geq 1$ and $a \in F^\times$, there is some monic irreducible polynomial over $F$ of degree $n$ with constant term $a$.
\end{remark}

\bibliography{biblio}{}

\begin{bibdiv}
\begin{biblist}

\bib{ARD}{article}{
      author={Ardakov, K.},
       title={Equivariant {$\sD$}-modules on rigid analytic spaces},
        date={2021},
        ISSN={0303-1179,2492-5926},
     journal={Ast\'{e}risque},
      number={423},
       pages={161},
         url={https://doi.org/10.24033/ast},
      review={\MR{4234536}},
}

\bib{AW1}{article}{
      author={Ardakov, K.},
      author={Wadsley, S.},
       title={{$\wideparen{\sD}$}-modules on rigid analytic spaces {I}},
        date={2019},
        ISSN={0075-4102},
     journal={J. Reine Angew. Math.},
      volume={747},
       pages={221\ndash 275},
         url={https://doi.org/10.1515/crelle-2016-0016},
      review={\MR{3905134}},
}

\bib{AW2}{article}{
      author={Ardakov, K.},
      author={Wadsley, S.},
       title={Equivariant line bundles with connection on the $p$-adic upper
  half plane},
        date={2023},
     journal={arXiv: 2309.05462},
        note={To appear in Algebra and Number Theory},
}

\bib{AW}{article}{
      author={Ardakov, K.},
      author={Wadsley, S.},
       title={Global sections of equivariant line bundles on the $p$-adic upper
  half plane},
        date={2023},
     journal={arXiv: 2312.12395},
}

\bib{CSA}{book}{
      author={Berhuy, G.},
      author={Oggier, F.},
       title={An introduction to central simple algebras and their applications
  to wireless communication},
      series={Mathematical Surveys and Monographs},
   publisher={American Mathematical Society, Providence, RI},
        date={2013},
      volume={191},
        ISBN={978-0-8218-4937-8},
         url={https://doi.org/10.1090/surv/191},
        note={With a foreword by B. A. Sethuraman},
      review={\MR{3086869}},
}

\bib{BHLTL}{article}{
      author={Bushnell, C.~J.},
      author={Henniart, G.},
       title={Local tame lifting for {${\rm GL}(N)$}. {I}. {S}imple
  characters},
        date={1996},
        ISSN={0073-8301,1618-1913},
     journal={Inst. Hautes \'Etudes Sci. Publ. Math.},
      number={83},
       pages={105\ndash 233},
         url={http://www.numdam.org/item?id=PMIHES_1996__83__105_0},
      review={\MR{1423022}},
}

\bib{BHJL}{article}{
      author={Bushnell, C.~J.},
      author={Henniart, G.},
       title={The essentially tame {J}acquet-{L}anglands correspondence for
  inner forms of {${\rm GL}(n)$}},
        date={2011},
        ISSN={1558-8599,1558-8602},
     journal={Pure Appl. Math. Q.},
      volume={7},
      number={3},
       pages={469\ndash 538},
         url={https://doi.org/10.4310/PAMQ.2011.v7.n3.a2},
      review={\MR{2848585}},
}

\bib{CHR}{book}{
      author={Chase, S.~U.},
      author={Harrison, D.~K.},
      author={Rosenberg, A.},
       title={Galois theory and cohomology of commutative rings},
   publisher={American Mathematical Soc.},
        date={1969},
      volume={52},
}

\bib{CDN1}{article}{
      author={Colmez, P.},
      author={Dospinescu, G.},
      author={Nizio\l, W.},
       title={Cohomologie {$p$}-adique de la tour de {D}rinfeld: le cas de la
  dimension 1},
        date={2020},
        ISSN={0894-0347},
     journal={J. Amer. Math. Soc.},
      volume={33},
      number={2},
       pages={311\ndash 362},
         url={https://ezproxy-prd.bodleian.ox.ac.uk:2102/10.1090/jams/935},
      review={\MR{4073863}},
}

\bib{CON}{article}{
      author={Conrad, B.},
       title={Irreducible components of rigid spaces},
        date={1999},
        ISSN={0373-0956,1777-5310},
     journal={Ann. Inst. Fourier (Grenoble)},
      volume={49},
      number={2},
       pages={473\ndash 541},
         url={http://www.numdam.org/item?id=AIF_1999__49_2_473_0},
      review={\MR{1697371}},
}

\bib{DLB}{article}{
      author={Dospinescu, G.},
      author={Le~Bras, A.-C.},
       title={Rev\^{e}tements du demi-plan de {D}rinfeld et correspondance de
  {L}anglands {$p$}-adique},
        date={2017},
        ISSN={0003-486X},
     journal={Ann. of Math. (2)},
      volume={186},
      number={2},
       pages={321\ndash 411},
  url={https://ezproxy-prd.bodleian.ox.ac.uk:2102/10.4007/annals.2017.186.2.1},
      review={\MR{3702670}},
}

\bib{JUNDR}{article}{
      author={Junger, D.},
       title={Cohomologie de de {R}ham du rev\^etement mod\'er\'e de l'espace
  de {D}rinfeld},
        date={2024},
        ISSN={1474-7480,1475-3030},
     journal={J. Inst. Math. Jussieu},
      volume={23},
      number={6},
       pages={2733\ndash 2775},
         url={https://doi.org/10.1017/S1474748024000082},
      review={\MR{4825127}},
}

\bib{KOH}{article}{
      author={Kohlhaase, J.},
       title={Lubin-{T}ate and {D}rinfeld bundles},
        date={2011},
        ISSN={0040-8735},
     journal={Tohoku Math. J. (2)},
      volume={63},
      number={2},
       pages={217\ndash 254},
  url={https://doi-org.ezproxy-prd.bodleian.ox.ac.uk/10.2748/tmj/1309952087},
      review={\MR{2812452}},
}

\bib{LANG}{book}{
      author={Lang, S.},
       title={Algebra},
     edition={third},
      series={Graduate Texts in Mathematics},
   publisher={Springer-Verlag, New York},
        date={2002},
      volume={211},
        ISBN={0-387-95385-X},
         url={https://doi.org/10.1007/978-1-4613-0041-0},
      review={\MR{1878556}},
}

\bib{LINDEN}{article}{
      author={Linden, G.},
       title={Equivariant vector bundles on the {D}rinfeld upper half space
  over a local field of positive characteristic},
        date={2023},
     journal={arXiv: 2304.03166},
}

\bib{LIT}{article}{
      author={Litoff, O.},
       title={On the commutator subgroup of the general linear group},
        date={1955},
     journal={Proceedings of the American Mathematical Society},
      volume={6},
      number={3},
       pages={465\ndash 470},
}

\bib{MONT2}{book}{
      author={Montgomery, S.},
       title={Fixed rings of finite automorphism groups of associative rings},
      series={Lecture Notes in Mathematics},
   publisher={Springer, Berlin},
        date={1980},
      volume={818},
        ISBN={3-540-10232-9},
      review={\MR{590245}},
}

\bib{ORL}{article}{
      author={Orlik, S.},
       title={Equivariant vector bundles on {D}rinfeld's upper half space},
        date={2008},
        ISSN={0020-9910,1432-1297},
     journal={Invent. Math.},
      volume={172},
      number={3},
       pages={585\ndash 656},
         url={https://doi.org/10.1007/s00222-008-0112-3},
      review={\MR{2393081}},
}

\bib{PRR}{book}{
      author={Platonov, V.},
      author={Rapinchuk, A.},
      author={Rapinchuk, I.},
       title={Algebraic groups and number theory. {V}ol. {I}},
     edition={Second},
      series={Cambridge Studies in Advanced Mathematics},
   publisher={Cambridge University Press, Cambridge},
        date={2023},
      volume={205},
        ISBN={978-0-521-11361-8},
      review={\MR{4615820}},
}

\bib{RAPVIE}{article}{
      author={Rapoport, M.},
      author={Viehmann, E.},
       title={Towards a theory of local {S}himura varieties},
        date={2014},
        ISSN={1867-5778,1867-5786},
     journal={M\"unster J. Math.},
      volume={7},
      number={1},
       pages={273\ndash 326},
      review={\MR{3271247}},
}

\bib{SchStu}{article}{
      author={Schneider, P.},
      author={Stuhler, U.},
       title={The cohomology of {$p$}-adic symmetric spaces},
        date={1991},
        ISSN={0020-9910,1432-1297},
     journal={Invent. Math.},
      volume={105},
      number={1},
       pages={47\ndash 122},
         url={https://doi.org/10.1007/BF01232257},
      review={\MR{1109620}},
}

\bib{STALG}{article}{
      author={Schneider, P.},
      author={Teitelbaum, J.},
       title={Algebras of {$p$}-adic distributions and admissible
  representations},
        date={2003},
        ISSN={0020-9910,1432-1297},
     journal={Invent. Math.},
      volume={153},
      number={1},
       pages={145\ndash 196},
         url={https://doi.org/10.1007/s00222-002-0284-1},
      review={\MR{1990669}},
}

\bib{SW}{article}{
      author={Scholze, P.},
      author={Weinstein, J.},
       title={Moduli of {$p$}-divisible groups},
        date={2013},
        ISSN={2168-0930,2168-0949},
     journal={Camb. J. Math.},
      volume={1},
      number={2},
       pages={145\ndash 237},
         url={https://doi.org/10.4310/CJM.2013.v1.n2.a1},
      review={\MR{3272049}},
}

\bib{TAY3}{article}{
      author={Taylor, J.},
       title={Equivariant vector bundles with connection on {D}rinfeld
  symmetric spaces},
        date={2024},
     journal={arXiv: 2406.14543},
        note={To appear in Algebra and Number Theory},
}

\bib{TAY2}{article}{
      author={Taylor, J.},
       title={Line bundles on the first {D}rinfeld covering},
        date={2024},
        ISSN={2491-6765},
     journal={\'Epijournal G\'eom. Alg\'ebrique},
      volume={8},
       pages={Art.\ 16},
      review={\MR{4864398}},
}

\bib{TAYSchur}{article}{
      author={Taylor, J.},
       title={Character theory for semilinear representations},
        date={2025},
     journal={arXiv: 2511.04296},
}

\end{biblist}
\end{bibdiv}
\bibliographystyle{plain}
\vspace{2em}
\end{document}